\documentclass[a4paper,12pt]{article}
\usepackage[utf8]{inputenc}
\usepackage{subfigure}
\usepackage{amsmath, epsfig,amssymb,amsthm,color}
\usepackage{caption}
\usepackage{algorithm}
\usepackage{algorithmic}
\usepackage{xypic}
\usepackage{multirow}
\usepackage{mathrsfs}
\usepackage{color}
\usepackage{eucal}%    caligraphic-euler fonts: \mathcal{ }
\usepackage{eufrak}%
\usepackage[round]{natbib}

\newtheorem{theorem}{Theorem}[section]
\newtheorem{proposition}[theorem]{Proposition}

\newtheorem{corollary}[theorem]{Corollary}
\newtheorem{definition}[theorem]{Definition}
\newtheorem{lemma}[theorem]{Lemma}

\newcommand{\NEW}[1]{\begingroup\color{blue}#1\endgroup}
\begin{document}
\title{Shapes of interacting RNA complexes}

\author{Benjamin M. M. Fu$^1$ \ Christian M. Reidys$^2$\footnotemark \\
Department of Mathematics and Computer Science,\\
 University of Southern Denmark,\\
 Campusvej 55, DK-5230 \\
Odense M, Denmark\\
Email$^1$: benjaminfmm@imada.sdu.dk\\
number$^1$: 45-40485667\\
Email$^{2*}$: duck@santafe.edu\\
number$^{2*}$: 45-24409251\\
Fax$^{2*}$: 45-65502325
}
\date{}
\maketitle
\footnotetext[2]{Corresponding author.}
\newpage
\begin{abstract}
Shapes of interacting RNA complexes are studied using a filtration 
via their topological genus. A shape of an RNA complex is obtained
by (iteratively) collapsing stacks and eliminating hairpin loops. 
This shape-projection preserves the topological core of the RNA 
complex and for fixed topological genus there are only finitely
many such shapes.
Our main result is a new bijection that relates the shapes
of RNA complexes with shapes of RNA structures.
This allows to compute the shape polynomial of RNA complexes 
via the shape polynomial of RNA structures.
We furthermore present a linear time uniform sampling algorithm 
for shapes of RNA complexes of fixed topological genus. \\
{\bf  keywords:} Interacting RNA complexes  \and 
\  Shape polynomials \and \ Topological genus \and \ Bijection 
\and Uniform generation
% {\bf   2010 MSC: 05A19 \and 92E10 }
\end{abstract}

\newpage
%%%
%%%%%%%%%%%%%%%%%%%%%%%%%%%%%%%%%%%%%%%%%%%%%%%%%%%%%%%%%%%%%%%%%%%%%%%%%%
%%%
\section{Introduction}
%%%
%%%%%%%%%%%%%%%%%%%%%%%%%%%%%%%%%%%%%%%%%%%%%%%%%%%%%%%%%%%%%%%%%%%%%%%%%
%%%

In this paper we study shapes of RNA complexes, which
constitute one of the fundamental mechanisms of cellular regulation. 
We find such interactions in a variety of contexts: 
small RNAs binding a larger (m)RNA target including: the regulation 
of translation in both prokaryotes \citet{Vogel:07} and eukaryotes 
\citet{McManus,Banerjee}, the targeting of chemical modifications 
\citet{Bachellerie}, insertion editing \citet{Benne} and transcriptional 
control \citet{Kugel}. RNA-RNA interactions are far more complex 
than simple sense-antisense interactions.
This is observed for a vast variety of RNA classes including miRNAs,
siRNAs, snRNAs, gRNAs, and snoRNAs.

An RNA molecule is a linearly oriented sequence of four types of nucleotides,
namely,
{\bf A}, {\bf U}, {\bf C}, and {\bf G}. This sequence is endowed with a
well-defined orientation from the $5'$- to the $3'$-end and referred to as
the backbone.
Each nucleotide can form a base pair by interacting with at most one other
nucleotide by establishing hydrogen bonds. Here we restrict ourselves to
Watson-Crick base pairs {\bf GC} and {\bf AU} as well as the wobble base
pairs {\bf GU}. In the following, base triples as well as other types of more
complex interactions are neglected.

RNA structures can be presented as diagrams by drawing the backbone
horizontally and all base pairs as arcs in the upper half-plane,
see Fig.~\ref{F:RNAp}.
This set of arcs provides our coarse-grained RNA structure, 
ignoring any spatial embedding or geometry of the molecule
beyond its base pairs.
%%%%%%%%%
%%%%%%%%%%%%%%%%%%%%%%%%%%%%%%%%%%%%%%%%%%%%%%%%%%
%%%%%%%%%

%%%%%%%%
 \begin{figure}[ht]
 \begin{center}
 \includegraphics[width=0.9\textwidth]{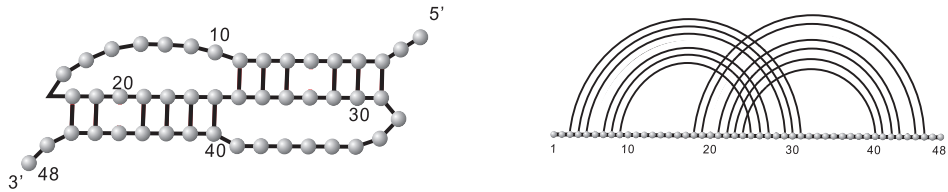}
 \end{center}
 \caption{\small (A) An RNA secondary structure and its diagram
 representation (B).
 }\label{F:RNAp}
 \end{figure}

%%%%%%%%%
%%%%%%%%%%%%%%%%%%%%%%%%%%%%%%%%%%%%%%%%%%%%%%%%%%
%%%%%%%%%
As a result, specific classes of base pairs translate into distinct
structure categories, the most prominent of which being secondary
structures \citet{Kleitman:70,Nussinov:1978,Waterman:79a,waterman1978secondary}.
Represented as diagrams, secondary structures have only non-crossing
base pairs (arcs).
Beyond RNA secondary structures we find RNA pseudoknot structures.
These exhibit cross serial interactions \citet{Rivas:99}. 
Once such cross serial interactions are considered the question of 
a meaningful filtration arises and to establish a relation to the
well-studied RNA secondary structures.

It turns out that topological genus is such a meaningful observable.
The genus of pseudoknotted, single stranded RNA has been studied in
\citet{vernizzi, vernizzib,bon,andersen2011} and there are several 
alternative filtrations of cross-serial interactions 
\citet{Orland:02,Reidys:11a,Reidys:10w}.

The objects studied here are derived from RNA complexes, that 
are diagrams over two backbones. Distinguishing
internal and external arcs, the former being arcs within one backbone and
the latter connecting the backbones, RNA complexes can be
represented by drawing the two backbones on top of each other, see
Fig.~\ref{F:diag_represent}.
%%%%%%%%%
%%%%%%%%%%%%%%%%%%%%%%%%%%%%%%%%%%%%%%%%%%%%%%%%%%
%%%%%%%%%

%%%%%%%%%
 \begin{figure}[ht]
 \begin{center}
 \includegraphics[width=0.9\textwidth]{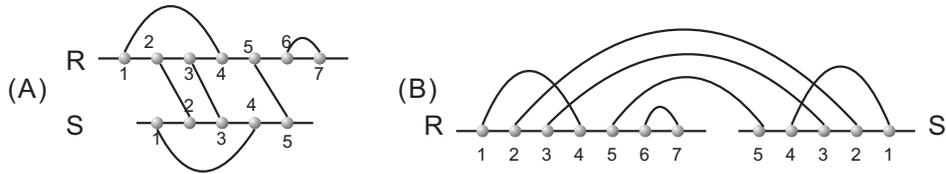}
 \end{center}
 \caption{\small Diagram representation of an RNA complex.
 }\label{F:diag_represent}
 \end{figure}

%%%%%%%%%
%%%%%%%%%%%%%%%%%%%%%%%%%%%%%%%%%%%%%%%%%%%%%%%%%%
%%%%%%%%%

We shall study shapes of RNA complexes, which are obtained by
recursively removing all arcs of length one and collapsing all 
parallel arcs, see Fig.~\ref{F:shape-map}. 
%%%%%%%%%
 \begin{figure}[ht]
 \begin{center}
 \includegraphics[width=0.9\textwidth]{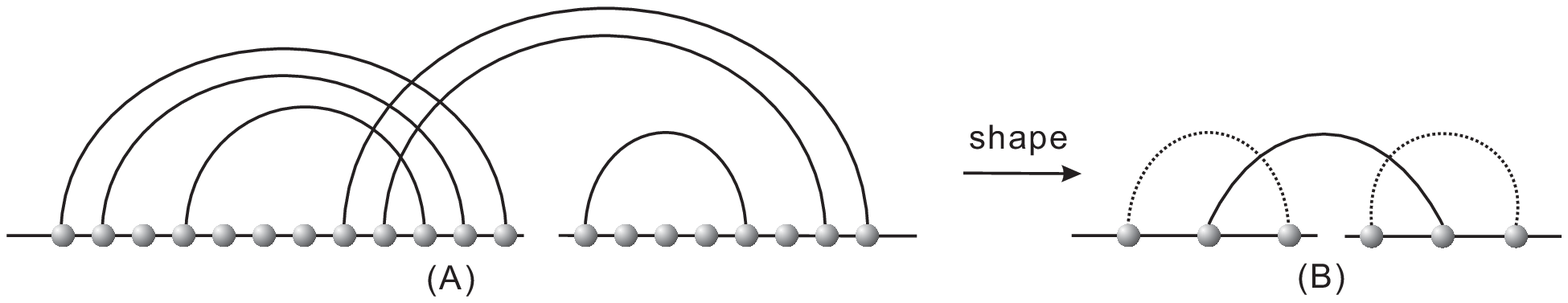}
 \end{center}
 \caption{\small From a $2$-backbone diagram to its shape.  The
 dashed arcs represent the rainbows (plants) of the shape.
 }\label{F:shape-map}
 \end{figure}

%%%%%%%%%
%%%%%%%%%%%%%%%%%%%%%%%%%%%%%%%%%%%%%%%%%%%%%%%%%%
%%%%%%%%%

Shapes are tailored to preserve the topological information of the molecule.
The particular topologization is obtained via the notion of fat graphs, which 
date back to Cayley. The classification and expansion of pseudoknotted RNA 
structures in terms of topological genus of a fat graph or double line graph
were first proposed by \citet{Orland:02} and \citet{Bon:08}.
In the context of RNA secondary structures, fat graphs were employed
even earlier in \citet{Waterman:93} and \citet{Penner:03}.  
The results of \citet{Orland:02} are based on the matrix models and
are conceptually independent. Genus, as well as other 
topological invariants of fat graphs were introduced and studied as 
descriptors of proteins in \citet{protein}.

The approach undertaken here is combinatorial and follows \citet{Fenix:t2}: 
starting with the diagram representation we inflate each edge, including 
backbone edges, into ribbons. As each ribbon has
two sides and specifying a counter-clockwise rotation around each vertex,
we obtain so called boundary cycles with a unique orientation. It is clear
that we have thus constructed a surface and its topological genus
provides the desired filtration. Naturally there are many such ribbon graphs
that produce the same topological surface (by gluing the two ``complementary''
sides of each ribbon), this is how we obtain the desired equivalence
(complexity) classes of structures.

It is easy to see that transforming an interaction structure into its shape
preserves topological genus and in Lemma~\ref{L:finite} we shall see that
for fixed genus $g$ there exist only finitely many such shapes of RNA 
complexes. This means that for fixed genus there are only finitely many
topologically distinct configurations and important information is captured
in the generating polynomial. In Theorem~\ref{T:shap2} we shall compute this
polynomial and relate its coefficients to shapes of RNA structures by means
of bijections relating one and two backbone shapes.

In \citet{fenix-shape} a linear time algorithm for uniformly generating shapes of
RNA structures of fixed topological genus was given. By means of the bijection 
of Theorem~\ref{T:main} relating one and two backbone shapes we can use this
algorithm to generate uniformly shapes of RNA complexes.

The paper is organized as follows: in Section~\ref{S:2} we introduce diagrams
and the basic framework in which we formulate our results. We discuss fat graphs 
and the topological filtration namely as drawing these diagrams on orientable
surfaces of higher topological genus. In Section~\ref{S:3} we develop the concept
of shapes and establish basic properties. We recall some key results on shapes of
RNA structures, in particular the two term recursion for computing their 
coefficients. In Section~\ref{S:4} we analyse shapes of RNA complexes
and relate them to shapes of RNA structures. Several constructions
show how to derive one from the other by specific ``shape-surgery''. Here we also
present the uniform generation algorithm of shapes of RNA complexes of fixed
topological genus.
In Section~\ref{S:5} we discuss specific RNA complexes, that all have a fixed
shape and in Section~\ref{S:6} we integrate and discuss our results. 

\newpage
%%%
%%%%%%%%%%%%%%%%%%%%%%%%%%%%%%%%%%%%%%%%%%%%%%%%%%%%%%%%%%%%%%%%%%%%%%%%%
%%%
\section{Some basic facts}\label{S:2}
%%%
%%%%%%%%%%%%%%%%%%%%%%%%%%%%%%%%%%%%%%%%%%%%%%%%%%%%%%%%%%%%%%%%%%%%
\begin{definition}
A diagram is a labeled graph over the vertex set $[n]=\{1,2,\ldots,n\}$ 
represented by drawing the vertices $1,2,\ldots,n$ on a horizontal line 
in the natural order and the arcs $(i,j)$, where $i<j$, in the upper 
half-plane.
The backbone of a diagram is the sequence of
consecutive integers $(1,\dots,n)$ together with the edges $\{\{i,i+1\}
\mid 1\le i\le n-1\}$.  A diagram over $b$ backbones is a diagram
together with a partition of $[n]$ into $b$ backbones, 
see Fig.~\ref{F:diagram}.
\end{definition}
 %%%%%%%%%%%%%%%%%%%%%%%%%%%%%%%%%%%%%%%%%%%%%%%%%%
 %%%%%%%%%
 \begin{figure}[ht]
 \begin{center}
 \includegraphics[width=0.7\textwidth]{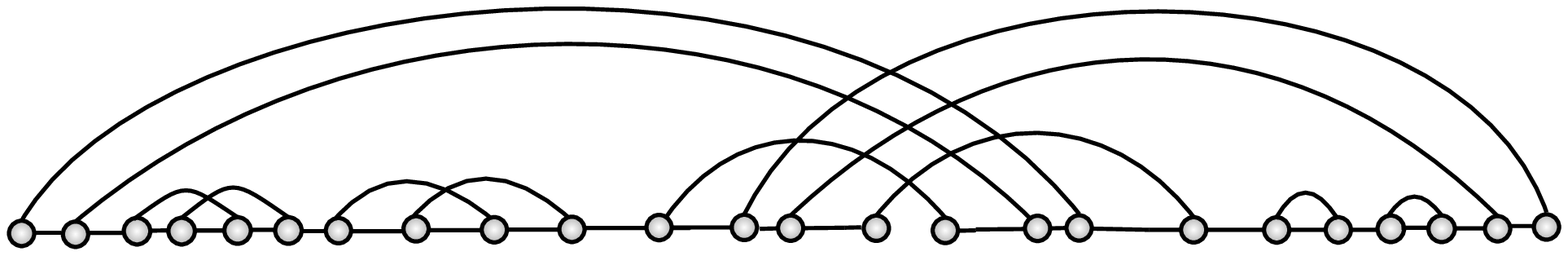}
 \end{center}
 \caption{A $2$-backbone diagram with $24$ vertices and $12$ arcs.
 }\label{F:diagram}
 \end{figure}
 %%%%%%%%%
 %%%%
We shall distinguish backbone edges $\{i,i+1\}$ from arcs $(i,i+1)$, 
which we refer to as $1$-arcs. Two arcs $(i,j)$, $(r,s)$, where $i<r$ 
are crossing if $i<r<j<s$ holds.
Parallel arcs of the form $\{(i,j), (i+1,j-1), \cdots, (i+\ell-1, j-\ell+1)\}$
are called a stack, and $\ell$ is called the length of the stack.
A stack on $[i, j]$ of length $k$ naturally induces $(k-1)$ pairs of intervals 
of the form $([i+l, i+l+1],[j-l-1, j-l])$ where $0\leq l \leq k-2$. Any of 
these $2(k-1)$ intervals is referred to as a \emph{$P$-interval}. 
An interval $[i, i +1]$ is called a \emph{gap} if there exists a pair of 
subsequent backbones $B_1$ and $B_2$ such that $i(i+1)$ is the 
rightmost(leftmost) vertex of $B_1(B_2)$. The vertex $i$ is referred to as 
\emph{cut vertex}. Any interval other than a gap or $P$-interval is called a 
\emph{$\sigma$-interval}. Clearly, a diagram over $[n]$ contains $(n-1)$ 
intervals of length 1 and we distinguish three types: gap intervals, 
$P$-intervals and $\sigma$-intervals; see Fig.~\ref{F:backint}.
%%%
%%%%%%%%%%%%%%%%%%%%%%%%%%%%%%%%%%%%%%%%%%%%%%%%%%%%%%%%%%%%%%%%%%%%%%%%%%%%%%%%%%%
%%%
\NEW{
\begin{figure}[ht]
\begin{center}
\includegraphics[width=0.7\textwidth]{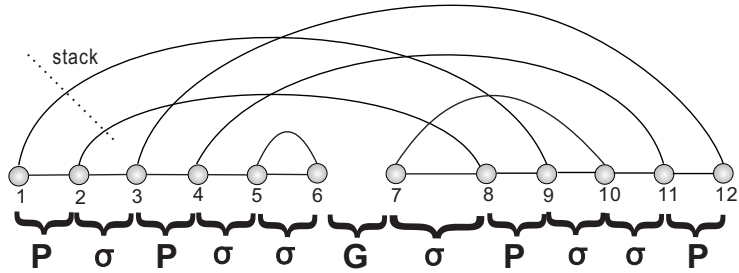}
\end{center}
\caption{\small Stacks and intervals:
gap intervals, $\sigma$-intervals and  $P$-intervals labelled  
by { G}, { $\sigma$} and {P}. 
There are $4$ stacks: $\{(1,9),(2,8)\}$, $\{(3,12),(4,11)\}$, $\{(5,6)\}$ and
$\{(7,10)\}$. 
}\label{F:backint}

\end{figure}}
%%%
%%%%%%%%%%%%%%%%%%%%%%%%%%%%%%%%%%%%%%%%%%%%%%%%%%%%%%%%%%%%%%%%%%%%%%%%%%%%%
%%%

Vertices and arcs of a diagram correspond to nucleotides and base pairs, respectively.
For a diagram over $b$ backbones, the leftmost vertex of each back-bone denotes the
$5'$ end of the RNA sequence, while the rightmost vertex denotes the $3'$ end.
The particular case $b=2$ is referred to as RNA interaction structures or RNA complexes. 
RNA complexes are oftentimes represented alternatively by drawing the two backbones
on top of each other, see Fig.~\ref{F:top}.
\begin{figure}[ht]
\begin{center}
\includegraphics[width=0.9\textwidth]{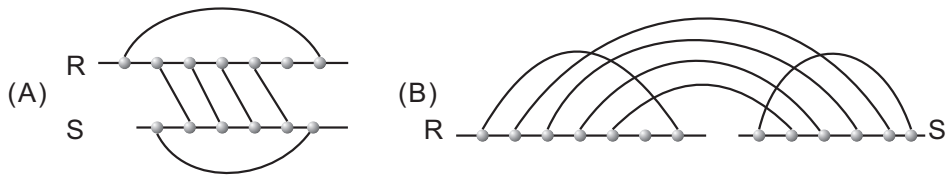}
\end{center}
\caption{\small LHS: An RNA complex presented by drawing the two backbones on top 
of each other. RHS: The corresponding diagram over two backbones.}
\label{F:top}
\end{figure}

We will add an additional ``rainbow-arc'' over each respective backbone
and refer to these diagrams as \textit{ planted  diagrams}, see Fig.~\ref{F:planted}.

 %%%%%%%%%%%%%%%%%%%%%%%%%%%%%%%%%%%%%%%%%%%%%%%%%%
 %%%%%%%%%
 \begin{figure}[ht]
 \begin{center}
 \includegraphics[width=0.8\textwidth]{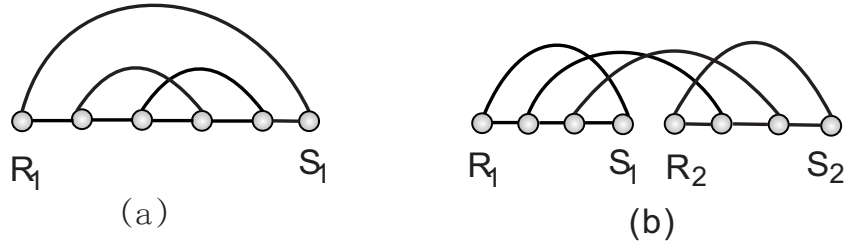}
 \end{center}
 \caption{(a) planted $1$-backbone diagram with the plant arc $(R_1,S_1)$,
 (b) planted $2$-backbone diagram with the plant arc $\{(R_1,S_1),(R_2,S_2)\}$.
 }\label{F:planted}
 \end{figure}
 %%%%%%%%%
 %%%%
A fat graph is a graph enriched by a cyclic ordering of the incident
half-edges at each vertex and consists of the following data: a set of
half-edges, $H$, cycles of half-edges as vertices and pairs of half-edges
as edges. The idea of half-edges stems from the observation that untwisted ribbons
have two sides and are traversed in complementary directions. 
It is then a matter of convention to denote the terminal half 
of these sides as half-edge, see Fig.~\ref{F:fat}.

The specific drawing of a diagram $G$ in the plane determines a cyclic 
ordering on the half edges of the underlying graph incident on each vertex, 
thus defining a corresponding fat graph $\mathbb{G}$. The collection of cyclic 
orderings is called fattening, one such ordering on the half-edges incident on 
each vertex, see Fig.~\ref{F:inflate}.

%%%%%%%%%%%%%%%%%%%%%%%%%%%%%%%%%%%%%%%%%%%%%%%%%%%%%%%%%%%%%%%%%%%%%%%%%%%%%%%%%%
%%%%%%%%%
 \begin{figure}[ht]
 \begin{center}
 \includegraphics[width=0.7\textwidth]{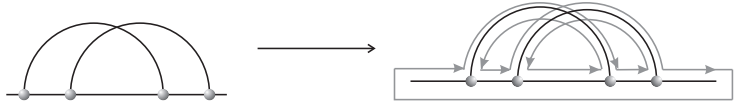}
 \end{center}
 \caption{\small The fattening.
 }\label{F:inflate}
 \end{figure}
 %%%%
 %%%%%%%%%%%%%%%%%%%%%%%%%%%%%%%%%%%%%%%%%%%%%%%%%%%%%%%%%%%%%%%%%%%%%%%%%%%%%%%%

A fat graph $\mathbb{G}$ can be embedded in a compact orientable surface $F(\mathbb{G})$, 
such that its complement is a disjoint union of simply connected domains (called the faces
or boundary components) 
and considered up to oriented homeomorphism. We can define the genus $g$ of the fat graph by 
the genus of the surface. Clearly, $F(\mathbb{G})$ contains $G$ as a deformation retract and 
each $\mathbb{G}$ represents a cell-complex \citet{Massey:69} over $F(\mathbb{G})$, see 
Fig.~\ref{F:fat}.

%%%%%%%%%%%%%%%%%%%%%%%%%%%%%%%%%%%%%%%%%%%%%%%%%%%%%%%%%%%%%%%%%%%%%%%%%%%%%%%%%%
%%%%%%%%%
 \begin{figure}[ht]
 \begin{center}
 \includegraphics[width=\textwidth]{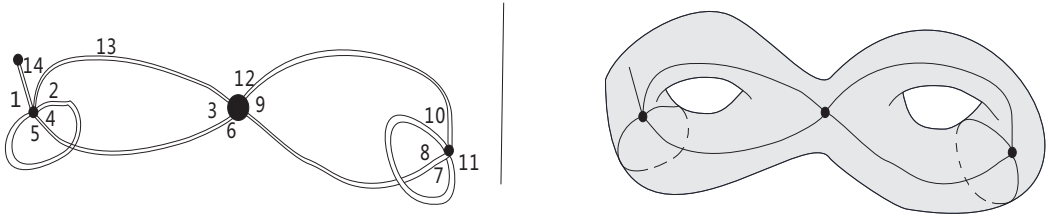}
 \end{center}
 \caption{\small A fatgraph and its embedding.
 }\label{F:fat}
 \end{figure}
 %%%%
 %%%%%%%%%%%%%%%%%%%%%%%%%%%%%%%%%%%%%%%%%%%%%%%%%%%%%%%%%%%%%%%%%%%%%%%%%%%%%%%%

A diagram $G$ hence determines a unique surface $F(\mathbb{G})$. 
Equivalence of simplicial and singular homology implies that
Euler characteristic $\chi$  and genus  $g$ of $F(\mathbb{G})$ 
are independent of the choice of the cell-complex $\mathbb{G}$
and given by $\chi = v-e+r$ and $g=1-\frac{1}{2}\chi$, where 
$v,e,r$ are the number of discs, ribbons and boundary components
in $\mathbb{G}$, respectively.

Without affecting topological type of the surface, one may collapse each backbone 
to a single vertex with the induced fattening called the polygonal model of the RNA, 
see Fig.~\ref{F:inflation}.
%%%%
%%%%%%%%%%%%%%%%%%%%%%%%%%%%%%%%%%%%%%%%%%%%%%%%%%%%%%%%%%%%%%%%%%%%%%%%%%%%%%%%%%
%%%%%%%%%
 \begin{figure}[ht]
 \begin{center}
 \includegraphics[width=0.9\textwidth]{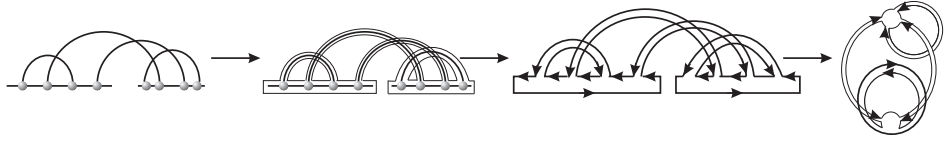}
 \end{center}
 \caption{\small Inflation of a $2$-backbone diagram and collapse of its $2$ backbones to  two vertices.
 }\label{F:inflation}
 \end{figure}
 %%%%
 %%%%%%%%%%%%%%%%%%%%%%%%%%%%%%%%%%%%%%%%%%%%%%%%%%%%%%%%%%%%%%%%%%%%%%%%%%%%%%%%
 %%%%
 
This backbone-collapse preserves orientation, Euler characteristic and genus. It is 
reversible by inflating each vertex to form a backbone. Using the 
collapsed fat graph representation, we see that for a connected 
diagram over $b$ backbones, the genus $g$ of the surface is 
determined by the number $n$ of arcs and the number 
$r$ of boundary components, namely, $2-2g-r=v-e=b-n$.
%%%
%%%%%%%%%%%%%%%%%%%%%%%%%%%%%%%%%%%%%%%%%%%%%%%%%%%%%%%%%%%%%%%%%%%%%%%%%%%%%%%%%%%
%%%

Boundary components are in the following oftentimes referred to as loops. 
We distinguish the following loop-types: 
\begin{itemize}
\item {\it hairpin loops}, which are boundary components of length one, 
\item {\it interior loops}, which are boundary components of length two,
\item {\it multi-loops}, which are boundary components of length two $\geq 3$. 
\end{itemize}
We furthermore distinguish within multiloops {\it pseudoknot loops}, which are
multi-loops containing some crossing arcs in the diagram representation.
In interaction structures, we shall distinguish 
{\it $\alpha$-loops} and {\it $\beta$-loops}, {\it $\alpha$ stacks} and {\it $\beta$ stacks},
depending on whether or not they contain only arcs whose endpoints are on one backbone.
\newpage
%%%
%%%%%%%%%%%%%%%%%%%%%%%%%%%%%%%%%%%%%%%%%%%%%%%%%%%%%%%%%%%%%%%%%%%%%%%%%%%%%%%%%%%
%%%
\section{Shapes}\label{S:3}
%%%
%%%%%%%%%%%%%%%%%%%%%%%%%%%%%%%%%%%%%%%%%%%%%%%%%%%%%%%%%%%%%%%%%%%%%%%%%%%%%%%%%%%
%%%

A diagram is called a preshape if it contains neither $1$-arcs (the arcs has the form 
$(i,i+1)$) nor stacks (parallel arcs) and isolated vertices (the vertices not paired).  
A preshape without a rainbow is called pure.  A shape is then obtained from a pure preshape 
by adding a rainbow for every backbone, see Fig.~\ref{F:shapegenus10}. 
 %%%
  %%%%%%%%%%%%%%%%%%%%%%%%%%%%%%%%%%%%%%%%%%%%%%%%%%%%%%%%%%%%%%%%%%%%%%%%%%%%%%%%%%%%%%%
  %%%
  \begin{figure}[ht]
  \begin{center}
  \includegraphics[width=0.9\textwidth]{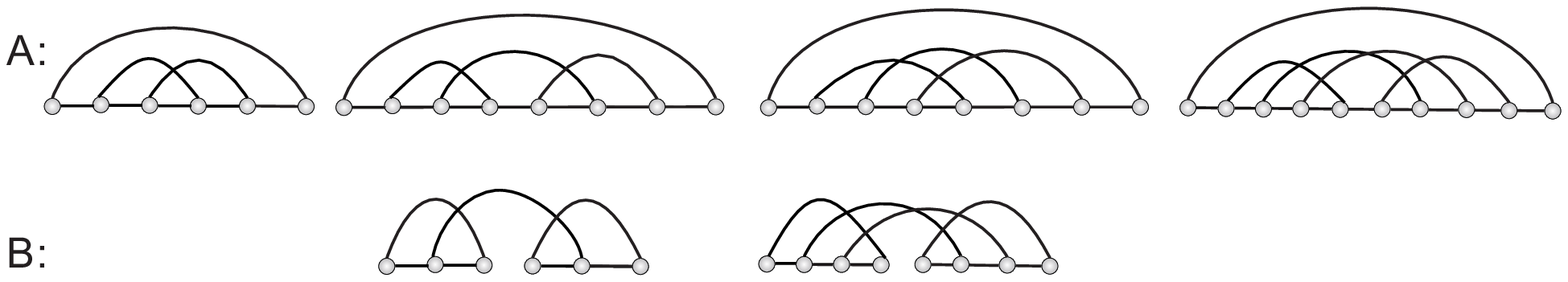}
  \end{center}
  \caption{\small 
   $A:$ the $4$ shapes of genus $1$ over one backbone.
    $B:$ the $2$ shapes of genus $0$ over two backbones.
  }\label{F:shapegenus10}
  \end{figure}
 %%%%
 %%%%%%%%%%%%%%%%%%%%%%%%%%%%%%%%%%%%%%%%%%%%%%%%%%%%%%%%%%%%%%%%%%%%%%%%%%%%%%%%%%%%%%%
We can obtain the shape of a planted diagram by iterating the following two steps: first 
collapse each stack into an arc, secondly remove all the $1$-arcs and isolated vertices. 
Iteration generates an unique diagram without stacks, $1$-arcs and isolated vertices, 
see Fig.~\ref{F:shape}.
 
 \begin{figure}[ht]
 \begin{center}
 \includegraphics[width=0.8\columnwidth]{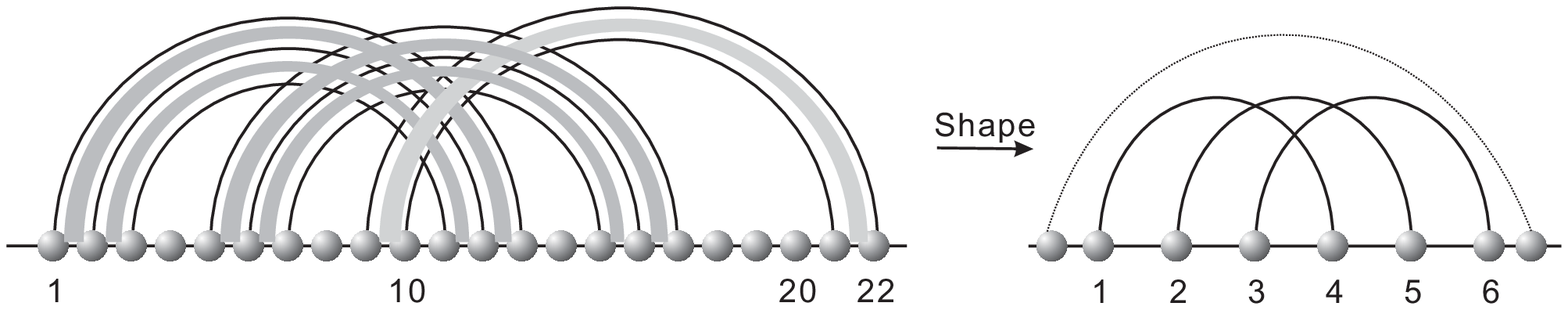}
 \end{center}
 \caption{\small From a diagram to a shape by removing all $1$-arc and parallel arcs.
 The dashed arc is a rainbow, displayed together with a nested preshape.
 }
 \label{F:shape}
 \end{figure}
%%%
%%%
%%%%%%%%%%%%%%%%%%%%%%%%%%%%%%%%%%%%%%%%%%%%%%%%%%%%%%%%%%%%%%%%%%%%%%%%%%%%%%%%%
%%%
For fixed genus $g$, there exist only finitely many shapes over $1$ backbone ($2$ backbones) 
\citet{ Fenix:t2,Reidys:11a}. 
%%%
%%%%%%%%%%%%%%%%%%%%%%%%%%%%%%%%%%%%%%%%%%%%%%%%%%%%%%%%%%%%%%%%%%%%%%%%%%%%%%%%%%%
%%%
\begin{lemma}\label{L:finite}
Given a $1$-backbone shape of genus $g$ with n edges, we have $2g+1 \leq n\leq 6g-1$. 
Therefore, for fixed genus $g$, there exist only finitely many shapes.
\end{lemma}
%%%
%%%%%%%%%%%%%%%%%%%%%%%%%%%%%%%%%%%%%%%%%%%%%%%%%%%%%%%%%%%%%%%%%%%%%%%%%%%%%%%%%%%
%%%
\begin{proof}
First note that if there is more than one boundary component, then there
must be an arc with different boundary components on its two sides and removing 
this arc decreases $r$ by exactly one while preserving $g$ since the number of 
arcs is given by $n=2g+r-1$. Furthermore, if there are $v_l$
boundary components of length $l$ in the polygonal model, then $2n =\sum_l l v_l$
since each side of each arc is traversed once by the boundary
(including the plant). 
For a shape, $v_1=1$, since the plant gives  the only boundary component of length $1$. 
$v_2=0$ by the definition of shapes.
It therefore follows that $2n =\sum_l l v_l\geq 3(r-1)+1$,  
so $2n=4g+2r-2\geq 3r-2,$ i.e., $4g \geq r$. 
Thus, we have $n=(2g+4g-1)=6g-1$, i.e., any shape can
contain at most $6g-1$ arcs. 
The lower bound $2g+1$ follows directly from $n=2g+r-1$ since $r \geq 2$.

For fixed genus $g$, the number of arcs in the shape is at most $6g-1$, the second assertion
follows.
\end{proof}

The lemma \ref{L:finite} implies that the generating function for a $1$-backbone shapes 
of genus $g$ is a polynomial. For example, for the shapes over $1$ backbone with genus 
$1$ to $3$, we have
\begin{eqnarray*}
S_1(z)& =& z^3+2z^4+z^5,\\
S_2(z)& =& 21 z^5+189 z^6+651 z^7+1134 z^8+1071 z^9+525 z^{10}+105 z^{11},\\
S_3(z)&=& 1485 z^7+26928 z^8+198451 z^9+808478 z^{10}+2054305 z^{11}\\
& &+3442340 z^{12} +3883363 z^{13}+2928926 z^{14}+1419418 z^{15} \\
& &+400400 z^{16}+50050 z^{17}
\end{eqnarray*}
Explicit formulas for the coefficients of the shape-polynomial of arbitrary fixed genus have
been given in \citet{fenix-shape}. There the {\it Poincar\'{e} dual} of shapes, a unicellular 
map was constructed and a construction of \citet{Chapuy:11} is refined to slice such a map 
into a tree with certain labeled vertices. The latter represent the blueprint to rebuild the
original unicellular map and the shape, respectively.

%%%
%%%%%%%%%%%%%%%%%%%%%%%%%%%%%%%%%%%%%%%%%%%%%%%%%%%%%%%%%%%%%%%%%%%%%%%%%%%%%%%%%%%%%%%%%%
%%%
\begin{theorem}\citet{fenix-shape}\label{T:1.1}
The shape polynomial of genus $g$ is given by
\begin{equation}
S_g(z)= \sum_{t=1}^{g} \kappa_t^{(g)} z^{2g+t-1}(1+z)^{2g+t-1},
\end{equation}
where $\kappa_t^{(g)}=a_{t-1}^{(g)} {\rm Cat}(2g+t)$ and
\begin{equation}\label{E:at}
a_t^{(g)}= \sum_{0=g_0<g_1<\cdots < g_r=g \atop 0=t_0=t_1\le t_2\le \cdots \le t_r=r-t}
\prod_{i=1}^r \frac{1}{2g_i} {2g+t-(2g_{i-1}+(i-1))+t_i \choose 2(g_i-g_{i-1})+1}.
\end{equation}
\end{theorem}
%%%
%%%%%%%%%%%%%%%%%%%%%%%%%%%%%%%%%%%%%%%%%%%%%%%%%%%%%%%%%%%%%%%%%%%%%%%%%%%%%%%%%%%%%%%%%%%
%%%

\citet{fenix-shape} furthermore derives from the underlying bijections a uniform 
generation algorithm \textit{UniformShape} for shapes of a fixed genus $g$, which
has linear time complexity.
   
\citet{thomas} studies the sequence $(\kappa_t^{(g)})_{t=1}^{g}$, see Tab.~\ref{T:kappa},
which emerged originally in the computation of the virtual Euler characteristic of a 
curve \citet{Harer}. 
%%%
%%%%%%%%%%%%%%%%%%%%%%%%%%%%%%%%%%%%%%%%%%%%%%%%%%%%%%%%%%%%%%%%%%%%%%%%%
%%%
%\begin{table}
%\begin{center}
%\begin{tabular}{c|ccccc}
% & g=1 & 2 & 3 & 4 & 5 \\
%\hline
 % t=0 & 1 & 21 & 1485 & 225225 & 59520825 \\
 % 1 &  & 105 & 18018 & 4660227 & 1804142340 \\
 % 2 &  &  & 50050 & 29099070 & 18472089636 \\
  %3 &  &  &  & 56581525 & 78082504500 \\
 % 4 &  &  &  &  &  117123756750
%\end{tabular}
%\end{center}
%\caption{\small The coefficients $\kappa_t^{(g)}$.
%}\label{T:kappa}
%\end{table}
%%%
%%%
%%%
%%%%%%%%%%%%%%%%%%%%%%%%%%%%%%%%%%%%%%%%%%%%%%%%%%%%%%%%%%%%%%%%%%%%%%%%%%%%%%
%%%
\citet{thomas} shows that $(\kappa_t^{(g)})_{t=1}^{g}$ is log-concave 
and hence unimodal and derives 
$$
\kappa_t^{(g)}=\frac{(2(2g+t-1))!}{2^{2g}(2g+t-1)!
\sum_{\gamma\vdash g}\prod_i m_i ! (2i+1)^{m_i}}.
$$
Furthermore,
%%%
%%%%%%%%%%%%%%%%%%%%%%%%%%%%%%%%%%%%%%%%%%%%%%%%%%%%%%%%%%%%%%%%%%%%%%%%%%%%%%
%%%
\begin{proposition}\citet{thomas}\label{P:recursion}
$\kappa_t^{(g)}$ satisfies
\begin{small}
 $$
(2g+t)\kappa_t^{(g)} \!=\! (2(2g+t)-3)(2(2g+t)-5)\left((2g+t-2)\kappa_t^{(g-1)}+
2(2(2g+t)-7)\kappa_{t-1}^{(g-1)}\right),
$$ 
\end{small}
where $\kappa_1^{(1)}=1$, $\kappa_t^{(g)}=0$, if $t<1$ or $t>g$. 
\end{proposition}

%%%%%%%%%%%%%%%%%%%%%%%%%%%%%%%%%%%%%%%%%%%%%%%%%%%%%%%%%%%%%%%%%%%%%%%%%
%%%

The above recursion has also been derived by \citet{chekhov} using matrix models.
\newpage
%%%
%%%%%%%%%%%%%%%%%%%%%%%%%%%%%%%%%%%%%%%%%%%%%%%%%%%%%%%%%%%%%%%%%%%%%%%%%%
%%%
\section{Shapes over two backbones}\label{S:4}
%%%
%%%%%%%%%%%%%%%%%%%%%%%%%%%%%%%%%%%%%%%%%%%%%%%%%%%%%%%%%%%%%%%%%%%%%%%%%%
%%%

In this section, we study shapes over two backbones. 
Our main observation is that shapes over two backbones
correspond to particular shapes over one backbone with topological
genus increased by one.
 
We denote a shape over one backbone by $(B,\alpha)$, 
where $$B:=[R_1,1,2, \cdots 2n, S_1]$$ is the sequence of 
vertices along the backbone and $\alpha$ is a fixed-point 
free involution, which contains $(R_1,S_1)$ as one cycle (rainbow). 
$\alpha$-cycles represent edges and $(R_1,S_1)$ is the plant.

We shall now distinguish two types of shapes. A shape is an $A$-shape 
if the vertex following $\alpha(1)$ is paired with the last vertex 
before $S_1$ and a $B$-shape, otherwise, see Fig.~\ref{F:fig13}.
%%%
%%%%%%%%%%%%%%%%%%%%%%%%%%%%%%%%%%%%%%%%%%%%%%%%%%%%%%%%%%%%%%%%%%%%%%%%%%%%%%%%%%%%%
%%%
 \begin{figure}[ht]
 \begin{center}
 \includegraphics[width=0.8\textwidth]{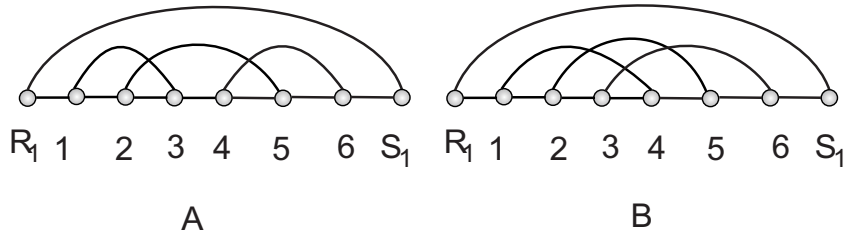}
 \end{center}
 \caption{ $A$-shapes ($\alpha(1)+1=4$ is paired with $6$) and $B$-shapes 
( $\alpha(1)+1=5$ is not paired with $6$). }\label{F:fig13}
 \end{figure}
 %%%
 %%%%%%%%%%%%%%%%%%%%%%%%%%%%%%%%%%%%%%%%%%%%%%%%%%%%%%%%%%%%%%%%%%%%%%%%%%%%%%%%%%%
 %%%
Let the set of $A$- and $B$-shapes having $n$ edges and genus $g$ be denoted by
$\mathcal{A}_g(n)$ and $\mathcal{B}_g(n)$, respectively. Furthermore, let
$\mathcal{A}_g=\bigcup_n\mathcal{A}_g(n)$ and 
$\mathcal{B}_g=\bigcup_n\mathcal{B}_g(n)$, 
$\mathcal{S}_g(n)=\mathcal{A}_g(n) \bigcup \mathcal{B}_g(n)$.
%%%
%%%%%%%%%%%%%%%%%%%%%%%%%%%%%%%%%%%%%%%%%%%%%%%%%%%%%%%%%%%%%%%%%%%%%%%%%%%%%%%%%%%%%
%%%
\begin{lemma}\label{L:bi}
We have a bijection:
$$
\theta: \mathcal{A}_g(n+2) \longrightarrow \mathcal{B}_g(n+1),
$$
i.e.~there exists a pairing $(x,\theta(x))$ associating to each $A$-shape
and its unique $B$-shape. In particular, 
$$
\mathcal{S}_g=\mathcal{A}_g\dot\cup \mathcal{B}_g
$$
and $\vert \mathcal{S}_g\vert/2=\vert \mathcal{A}_g\vert$.
\end{lemma}
%%%
%%%%%%%%%%%%%%%%%%%%%%%%%%%%%%%%%%%%%%%%%%%%%%%%%%%%%%%%%%%%%%%%%%%%%%%%%%%%%%%%%%%%%
%%%
\begin{proof}
Let $\Gamma=([R_1,1,2, \cdots 2n+1,2n+2, S_1], \alpha)$ be an $A$-shape having $n+2$ arcs,
containing the arc $(\alpha(1)+1, 2n+2)$. Since $\Gamma$ is a shape, there are no nested 
arcs or $1$-arcs, whence removal of $(\alpha(1)+1, 2n+2)$ maps an $A$-shape into a 
$B$-shape.

 Furthermore, as an $A$-shape, $\Gamma$ has a boundary component of size three, 
$\gamma_3$, traversing the sides of the rainbow, $(1,\alpha(1))$ and $(\alpha(1)+1, 2n+2)$. 
Let $\theta$ be the mapping defined by removing the arc $(\alpha(1)+1, 2n+2)$ together
with its incident vertices and subsequently relabeling of the remaining vertices.
Then $\theta$ decreases both: the number of boundary components, $r$ as well as 
the number of arcs $n+2$ by $1$. To see this we note that $(\alpha(1)+1, 2n+2)$ is
traversed by two distinct boundary components, $\gamma,\gamma_3$. Removing 
$(\alpha(1)+1, 2n+2)$ consequently merges $\gamma$ and $\gamma_3$, whence 
the number of boundary components decreases by one. Euler's characteristic equation,
$2-2g-r=1-(n+2)$ shows that $\theta$ preserves $g$. See Fig.~\ref{F:fig14}

%%%
%%%%%%%%%%%%%%%%%%%%%%%%%%%%%%%%%%%%%%%%%%%%%%%%%%%%%%%%%%%%%%%%%%%%%%%%
%%%
 \begin{figure}[ht]
 \begin{center}
 \includegraphics[width=0.8\textwidth]{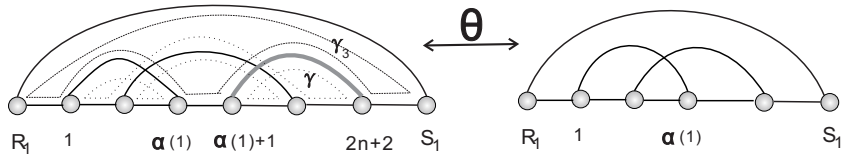}
 \end{center}
 \caption{ $\theta$: removal of $(\alpha(1)+1, 2n+2)$, creates a $B$-shape.}\label{F:fig14}
 \end{figure}
%%%
%%%%%%%%%%%%%%%%%%%%%%%%%%%%%%%%%%%%%%%%%%%%%%%%%%%%%%%%%%%%%%%%%%%%%%%%%%%%%
%%%

We next specify $\theta^{-1}$. Given $B$-shape having $n+1$ edges and genus $g$, 
we insert an arc with endpoints between $[\alpha(1),\alpha(1)+1]$ and $[2n,S_1]$ 
and subsequently relabel the diagram. This insertion maps any $B$-shape into 
an $A$-shape. Namely, by construction, it creates neither nested arcs nor 
$1$-arcs (the latter would imply that the rainbow has a nested arc).
After relabeling, the inserted arc is incident to $(\alpha(1)+1,2n+2)$ and 
creates a new boundary component, $\gamma_3$, as specified above. 
Euler's characteristic equation then shows that $\theta^{-1}$ does preserve
genus, see Fig.~\ref{F:fig14}.
\end{proof}

Let $\mathcal{Q}_g$ denote the set of shapes over two backbones of genus $g$ and
$\mathcal{S}_g^2$ denote the set of pairs of disconnected $1$-backbone shapes whose
sum of genera equals $g$. Let $\mathcal{Q}'_g=\mathcal{Q}_g\cup\mathcal{S}_g^2$.

%%%
%%%%%%%%%%%%%%%%%%%%%%%%%%%%%%%%%%%%%%%%%%%%%%%%%%%%%%%%%%%%%%%%%%%%%%%%%%%%%%%%%%%%%%%%%
%%%
\begin{theorem}\label{T:main}
We have the following commutative diagram of bijections
$$
\diagram
 \mathcal{Q}'_g\dline   \rrto^{\eta}   && \mathcal{A}_{g+1}\dline\\
 \mathcal{Q}'_g(n+2) \rrto^{\eta_n} &&  \mathcal{A}_{g+1}(n+3)
\enddiagram
$$
\end{theorem}
%%%
%%%%%%%%%%%%%%%%%%%%%%%%%%%%%%%%%%%%%%%%%%%%%%%%%%%%%%%%%%%%%%%%%%%%%%%%%%%%%%%%%%
%%%
\begin{proof}
Since any $\mathcal{Q}'_g$-diagram has a unique number of arcs it suffices to specify
the bijections $\eta_n$.

An $\mathcal{Q}'_g(n+2)$-element can be denoted by  
$$ 
x=([[R_1,1,2,\cdots, m,S_1],[R_2,m+1,\cdots 2n, S_2]], \alpha),
$$ 
having the rainbows $(R_1,S_1),(R_2,S_2)$.

We define the mapping $\eta_n$ as follows: 
\begin{itemize}
\item first we glue the two backbones into 
$$
[R_1,1,2,\cdots, m,S_1,R_2,m+1,\cdots 2n, S_2],
$$ 
\item secondly we add an new rainbow, 
\item thirdly, we relabel the vertices.
\end{itemize}
This produces a unique backbone 
$$
[R_1,1,2,\cdots,2n-1,2n,2n+1,2n+2,S_1]
$$
and transforms the two rainbows into the new arcs
$$
(R_1,S_1)\mapsto (1,\alpha(1))\quad \text{\rm and }\quad
(R_2,S_2)\mapsto (\alpha(1)+1,2n),
$$ 
respectively. Accordingly, $\eta_n(x)$ is an $A$-shape having $(n+3)$ edges,
 see Fig.~\ref{F:fig15}.

The mapping $\eta_n$ eliminates one backbone, i.e.~$b'=b-1$, generates a 
$\gamma_3$-boundary component merging the two original rainbow-boundaries 
and adds a new rainbow boundary, i.e.~$r'=r$ and adds one edge, i.e.~$n'=n+3$.
In view of $2-2g-r=2-(n+2)$ we obtain
$$ 
2g'= 2 - r - (2-1) + (n+3) = 2(g + 1),
$$
which proves that $\mathcal{A}_{g+1}(n+3)$.

We next construct $\eta_n^{-1}$ as follows:
consider an $A$-shape $y\in \mathcal{A}_{g+1}(n+3)$, then
\begin{itemize}
\item remove the rainbow,
\item cut the backbone between $\alpha(1)$ and $\alpha(1)+1$,
\item relabel the two respective backbones.
\end{itemize}
By construction the edges $(1,\alpha(1)),(\alpha(1)+1, 2n)$ become the rainbows
of the new backbones. 
The mapping $\eta_n^{-1}$ reverses $\eta_n$ and our above accounting of backbones, 
boundary components and edges applies here. Thus $\eta_n^{-1}(y)$ is a $2$-backbone
diagram of genus $g$ having $n+2$ edges, see Fig.~\ref{F:fig15}.
\end{proof}
%%%
%%%%%%%%%%%%%%%%%%%%%%%%%%%%%%%%%%%%%%%%%%%%%%%%%%%%%%%%%%%%%%%%%%%%
%%%
 \begin{figure}[ht]
 \begin{center}
 \includegraphics[width=0.8\textwidth]{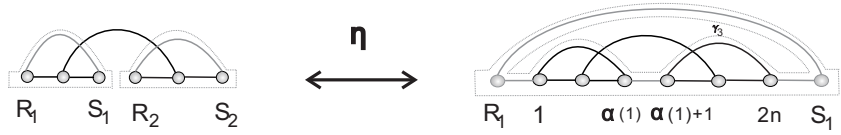}
 \end{center}
 \caption{The mapping $\eta$.}\label{F:fig15}
 \end{figure}
%%%
%%%%%%%%%%%%%%%%%%%%%%%%%%%%%%%%%%%%%%%%%%%%%%%%%%%%%%%%%%%%%
%%%
%%%
%%%%%%%%%%%%%%%%%%%%%%%%%%%%%%%%%%%%%%%%%%%%%%%%%%%%%%%%%%%%%
%%%
\begin{corollary}
Let $x\in \mathcal{Q}'_g(n+2)$ be a shape over two backbones containing 
$\ell$-multiloops, then $\eta_n(x)\in \mathcal{A}_{g+1}(n+3)$ is an 
$A$-shape over one backbone having $\ell+1$ multi-loops.
\end{corollary}
%%%
%%%%%%%%%%%%%%%%%%%%%%%%%%%%%%%%%%%%%%%%%%%%%%%%%%%%%%%%%%%%%
%%%
\begin{proof}
The map $\eta_n$ merges two rainbow-boundary components of $x$ 
and the new rainbow into a multi-loop of length $3$, see Fig.~\ref{F:fig15}.
\end{proof}
%%%
%%%%%%%%%%%%%%%%%%%%%%%%%%%%%%%%%%%%%%%%%%%%%%%%%%%%%%%%%%%%%%%%%%%%%%%%%%
%%% 
\begin{algorithm}[H]
\begin{algorithmic}[1]
\STATE {\tt UniformBi-shape}~($TargetGenus$)
\WHILE { $1$ }
\STATE {$\mathfrak{s}_1\leftarrow {\tt UnifromShape}(TargetGenus +1 )$}
\IF    {$\mathfrak{s}_1$ is type $A$}
\STATE {$\mathfrak{s}_2 \leftarrow {\tt \eta^{-1}(\mathfrak{s}_1)}$}
\ELSE
\STATE { $\mathfrak{s}_2 \leftarrow {\tt \eta^{-1} \theta^{-1} (\mathfrak{s}_1)}$ }
\ENDIF
\IF {{\tt Connection} ($\mathfrak{s}_2$)}
\STATE {\textbf{return} $\mathfrak{s}_2$}
\ENDIF
\ENDWHILE
\caption {\small Uniform genration of shapes over two backbones }
\label{A:bi}
\end{algorithmic}
\end{algorithm}
%%%
%%%%%%%%%%%%%%%%%%%%%%%%%%%%%%%%%%%%%%%%%%%%%%%%%%%%%%%%%%%%%%%%%%%%%%%%%
%%%
A first application of Theorem~\ref{T:main} is a uniform generation 
algorithm for shapes over two backbones of fixed topological genus $g$. We show the
pseudocode in Algorithm~\ref{A:bi}.

\begin{corollary}
Algorithm~\ref{A:bi} generates $2$-backbone shapes of genus $g$ uniformly.  
\end{corollary}
\begin{proof}
{\it UniformShape} \citet{fenix-shape} 
generates $1$-backbone shape uniformly and any $2$-backbone 
shape corresponds to either an $A$-shape via $\eta$, or a $B$-shape via $ 
 \theta  \circ \eta $. Since $A$ and $B$-shapes are generated uniformly, 
any two-backbone shape is generated uniformly with multiplicity two.
\end{proof}
%%%
%%%%%%%%%%%%%%%%%%%%%%%%%%%%%%%%%%%%%%%%%%%%%%%%%%%%%%%%%%%%%%%%%%%%%%%%%
%%%

Let $\mathcal{S}^2$ denote the set of pairs of disconnected shapes
whose sum of genera equals $g$ and let $s^2_g(n)$ denote the number
of these shapes having $n$ arcs. Then $S^2_g(z)= \sum_{n} s^2_g(n) z^n$ 
satisfies $S^{2}_g(z)= \sum\limits_1^{g}S_i(z)S_{g+1-i}(z)$.

%%%
%%%%%%%%%%%%%%%%%%%%%%%%%%%%%%%%%%%%%%%%%%%%%%%%%%%%%%%%%%%%%%%%%%%%%%%%%%%%%%%%
%%%
\begin{theorem}\label{T:shap2}
The polynomial of shapes of genus $g$ over two backbones, 
$Q_g(z)=\sum_{l}q_g(l) z^l$, is given by 
$Q_g(z)=Q'_g(z)-\sum_1^{g}S_i(z)S_{g+1-i}(z)$, 
where 
$$
Q'_g(z)=\frac{S_{g+1}(z)}{(1+z)} = 
\sum\limits_{t=1}^{g+1} \kappa_t^{(g+1)} z^{2g+t+1}(1+z)^{2g+t}.
$$ 
\end{theorem}
%%%
%%%%%%%%%%%%%%%%%%%%%%%%%%%%%%%%%%%%%%%%%%%%%%%%%%%%%%%%%%%%%%%%%%%%%
%%%%
\begin{proof}
Each $\mathcal{Q}'_g$-diagram is a $\mathcal{Q}'_g(n)$-diagram for a unique $n$.
As such we have
$$
\diagram
\mathcal{Q}'_g(n+2) \rrto^{\eta_n}\drrto^{\theta\circ \eta_n} &&  \mathcal{A}_{g+1}(n+3)\dto^{\theta}\\
 && \mathcal{B}_{g+1}(n+2)
\enddiagram
$$
Suppose the generating function of $A$- and $B$-shapes is 
$A_g(z)=\sum_n a_g(n)z^n$ and $B_g(z)=\sum_nb_g(n)z^n$, respectively. 
From the bijection $\theta:\mathcal{A}_{g+1}(n+3) \leftrightarrow \mathcal{B}_{g+1}(n+2)$, 
we obtain $b_{g+1}(n+2)=a_{g+1}(n+3)$. 
Then $S_{g+1}(z)=A_{g+1}(z)+ B_{g+1}(z)$ implies $S_{g+1}(z)= (1+1/z) A_{g+1}(z)$, or equivalently, 
$A_{g+1}(z)=\frac{S_{g+1}(z)}{1+1/z}$. By the bijection $\eta$, the generalized $2$-backbone 
shape $\mathfrak{s} \in \mathcal{Q}'_g$ has one arc less than $\eta(\mathfrak{s})$, 
which implies
$$
Q'_g(z)=\frac{A_{g+1}(z)}{z} = \frac{S_{g+1}(z)}{(1+z)}.
$$ 
Subtracting the set of disconnected $2$ backbone shapes, $S^2_g(z)$, the result follows.
\end{proof}
For genus $g=0,1,2$, we accordingly have
\begin{align*}
& {\bf Q}_0(z)=z^3+z^4\\
& {\bf Q}_1(z)=21 z^5+167 z^6+479 z^7+645 z^8+416 z^9+104 z^{10}\\
& {\bf Q}_2(z)= 1485 z^7 + 25401 z^8+ 172546 z^9+633370 z^{10} +1413585 z^{11} + 2015525 z^{12} \\
&\ \ \ \ \ \ \ \ \ \ +1852256 z^{13} +1064616 z^{14}+
348880 z^{15} + 49840 z^{16}\\
\end{align*}

\newpage
%%%
%%%%%%%%%%%%%%%%%%%%%%%%%%%%%%%%%%%%%%%%%%%%%%%%%%%%%%%%%%%%%%%%%%%%%%%%%%%%%%%%%%
%%%
\section{Fibers} \label{S:5}
%%%
%%%%%%%%%%%%%%%%%%%%%%%%%%%%%%%%%%%%%%%%%%%%%%%%%%%%%%%%%%%%%%%%%%%%%%%%%%%%%%%%%%
%%%

In the previous Section we computed the shape polynomials of shapes over
two backbones of fixed topological genus. Their coefficients can be recursively
determined and are directly related to the coefficients of polynomials of 
shapes over one backbone.

Furthermore, Theorem~\ref{T:main} implies a linear time sampling algorithm
for such $2$-backbone shapes of genus $g$. By means of their preimages, shapes
induce a natural partition of RNA complexes and here we shall study the 
sets of RNA complexes having a fixed shape, $\mathfrak{s}$,
to which we refer to as the fiber of $\mathfrak{s}$.

Given a $2$-backbone shape having $l$ arcs and genus $g$, $\mathfrak{s}_{g,l}$. 
Let $q_{\mathfrak{s}_{l,g}}(n)$ be the number of $2$-backbone matchings of genus $g$
having the shape $\mathfrak{s}_{l,g}$. 

%%%
%%%%%%%%%%%%%%%%%%%%%%%%%%%%%%%%%%%%%%%%%%%%%%%%%%%%%%%%%%%%%%%%%%%%%%%%%%%%%%%%%%%%%
%%%
\begin{theorem}
The generating function of matchings of genus $g$ having shape $\mathfrak{s}_{l,g}$
is given by 
$$
Q_{\mathfrak{s}_{l,g}}(z)=
\sum_n q_{\mathfrak{s}_l,g}(n)z^n={C_0}(z)^{2l+2} \frac{z^{l+2}}{(1-z C_0(z)^2)^{l+2}},
$$
where $C_0(z)= \frac{1-\sqrt{1-4z}}{2z}$.
In particular, the number of $2$-backbone
structures of length $n$ having genus $g$ and shape $\mathfrak{s}_{l.g}$
depends only on $l$ and  
$$
q_{\mathfrak{s}_l,g}(n) 
\sim \frac{k}{(l+1)!} n^{l+1} 4^{n-l-2}, 
$$
where $k$ is some positive constant.
\end{theorem}
%%%
%%%%%%%%%%%%%%%%%%%%%%%%%%%%%%%%%%%%%%%%%%%%%%%%%%%%%%%%%%%%%%%%%%%%%%%%%%%%%%%%%%%%%
%%%
%%%
%%%%%%%%%%%%%%%%%%%%%%%%%%%%%%%%%%%%%%%%%%%%%%%%%%%%%%%%%%%%%%%%%%%%%%%%%%%%%%%%%%%%%
%%
\begin{figure}[!ht]
 \begin{center}
 \includegraphics[width=0.6\textwidth]{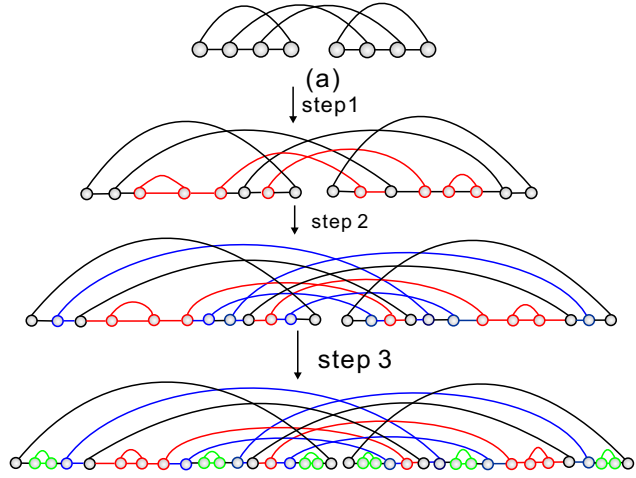}
 \end{center}
\caption{\small (a) a shape of genus $1$ with $4$ arcs; \protect \\
step 1: inflate each arc to a sequence of induced arcs (red);\protect \\
step 2: inflate each exterior arc to a stack (blue); \protect \\
step 3: insert a $\mathcal{C}_0$-matching into the $\sigma$-intervals (green).
}\label{F:insert}
 \end{figure}
%%%
%%%%%%%%%%%%%%%%%%%%%%%%%%%%%%%%%%%%%%%%%%%%%%%%%%%%%%%%%%%%%%%%%%%%%%%%%%%%%%%%%%%%%
%%

\begin{proof}
By the following steps, we can  inflate a RNA-complex from a shape.

Step $1$: we inflate each arc in $\mathfrak{s}_{l,g}$ into a sequence of induced arcs,
an induced arc $\mathcal{N}$ is an exterior arc together with at least 
one non-trivial genus $0$ matching in either one or both $P$-intervals, 
Clearly, we have ${N}(z)=z\left(2({ C_0}(z)-1)+({ C_0}(z)-1)^2\right)= z(C_0(z)^2-1)$. 
Furthermore, we inflate the arc into a sequence $\mathcal{M}$ 
of induced arcs ${M}(z)=\frac{1}{1-z\left({ C_0}(z)^2-1\right)}$.
Inflating all $l+2$ arcs (including the $2$ rainbows)
into a sequence of induced arcs, leads to
\begin{align*}
z^{l+2} {M}(z)^{l+2}=z^{l+2}\left(\frac{1}{1-z({C_0}(z)^2-1)}\right)^{l+2}.
\end{align*}
Denote the matching after this step by $x_1$.\\
Step $2$: we inflate each arc in $x_1$ into a stack. The corresponding 
generating function is 
\begin{align}
\left(\frac{\frac{z}{1-z}}{1-\frac{z}{1-z}
\left({C_0}(z)^2-1\right)}\right)^{l+2} =
\frac{z^{l+2}}{(1-z C_0(z)^2)^{l+2}}.
\end{align}
Step 3: we insert a $\mathcal{C}_0$ matching into the 
respective $(2l+2)$ $\sigma$-intervals of $\mathfrak{s}_{l.g}$. 
The corresponding generating function is ${C_0}(z)^{2l+2}$.

Combining the above three steps, we derive 
$$Q_{\mathfrak{s}_l,g}(z)=\sum_n q_{\mathfrak{s}_l,g}(n)z^n= 
{C_0}(z)^{2l+2} \frac{z^{l+2}}{(1-z C_0(z)^2)^{l+2}}, 
$$
where $q_{\mathfrak{s}_l,g}(n)$ denotes the number of genus $g$ matchings generated
from $\mathfrak{s}_{l,g}$.

The generating function has an unique, dominant singularity $\rho=1/4$ with multiplicity 
$l+2$. Standard singularity analysis \citet{flajolet}, implies
$$
q_{\mathfrak{s}_l,g}(n) \sim \frac{k}{(l+1)!} n^{l+1} 4^{n-l-2}.
$$
\end{proof}

%%%
%%%%%%%%%%%%%%%%%%%%%%%%%%%%%%%%%%%%%%%%%%%%%%%%%%%%%%%%%%%%%%%%%%%%%%%%%%%%%%%%%%%%%%%%%%%
%%%
\begin{corollary}
The generating function $W_g(z)$ of $2$-backbone matchings of genus $g$ is
given by
$$ 
W_g(z) = \sum_l q_g(l) Q_{\mathfrak{s_l},g}(z) = 
 \sum_l  q_g(l) {C_0}(z)^{2l+2} \frac{z^{l+2}}{(1-z C_0(z)^2)^{l+2}}.
$$
\end{corollary}
%%%
%%%%%%%%%%%%%%%%%%%%%%%%%%%%%%%%%%%%%%%%%%%%%%%%%%%%%%%%%%%%%%%%%%%%%%%%%%%%%%%%%%%%%%%%%%%
%%%
In particular we have $W_0(z)=\frac{z^3}{(1-4z)^2}$,
$$
W_1(z)= \frac{(20z+21)z^5}{(1-4z)^5},\quad
W_2(z)= \frac{(1696 z^2+6096 z+1485)z^7}{(1-4z)^8}.
$$
We conclude this section by discussing loops in shape-fibers. 
By construction, there are only multi-loops and pseudoknot-loops in a shape. 
We observe that the lengths of the original shape-loops increase in structures of 
the shape-fiber. Structures of the shape-fiber exhibit in addition hairpin loops, 
interior loops and two types of multi-loops, see Fig.~\ref{F:loops}. 

%%%
%%%
%%%%%%%%%%%%%%%%%%%%%%%%%%%%%%%%%%%%%%%%%%%%%%%%%%%%%%%%%%%%%%%%%%%%
%%%
 \begin{figure}[ht]
 \begin{center}
 \includegraphics[width=0.75\textwidth,height=0.35\textwidth]{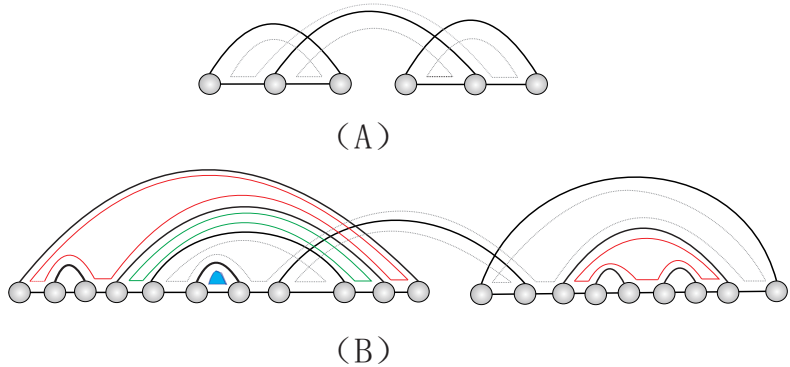}
 \end{center}
 \caption{A shape with a distinguished loop (A). 
          Inflation generates hairpin loops (blue), interior loops (green) and 
          two types of non-shape multi-loops (red) (B). 
          The length of the distinguished shape-loop increased by 
          $2$.}\label{F:loops}
 \end{figure}
%%%
%%%%%%%%%%%%%%%%%%%%%%%%%%%%%%%%%%%%%%%%%%%%%%%%%%%%%%%%%%%%%
%%%
\newpage
%%%%%%%%%%%%%%%%%%%%%%%%%%%%%%%%%%%%%%%%%%%%%%%%%%%%%%%%%%%%%%%%%%%%%%%%%%%%%%%%%%
%%%
\section{Discussion}\label{S:6}
%%%
%%%%%%%%%%%%%%%%%%%%%%%%%%%%%%%%%%%%%%%%%%%%%%%%%%%%%%%%%%%%%%%%%%%%%%%%%%%%%%%%%%
%%%

In this paper we study shapes of RNA complexes. We show that these shapes
are directly related to shapes of RNA structures of increased topological 
genus. More precisely, we show in Lemma~\ref{L:bi} that there is a bipartition 
of RNA-shapes into $A$-shapes and $B$-shapes. Furthermore, $A$- and $B$-shapes are
in one-to-one correspondence. We establish in Theorem~\ref{T:main} that each 
respective type is in one-to-one correspondence to shapes of RNA complexes.
These relations have various implications.

First Lemma~\ref{L:finite} guarantees that there are only finitely many such shapes. 
This leads to the shape polynomials for shapes of fixed topological genus $g$. The 
above correspondences
reduce the computation of the coefficients of these polynomials for shapes of RNA complexes 
to those of shapes of RNA structures. For the latter Proposition~\ref{P:recursion}
gives a simple two-term recursion, which allows us to obtain any such polynomials for shapes
of structures and complexes of fixed topological genus in constant time.

Secondly we obtain a sampling algorithm, Algorithm~\ref{A:bi} for shapes of RNA 
complexes that has linear time complexity. Algorithm \ref{A:bi} and the sampling algorithm of 
RNA shapes are freely available at
$$
\text{\tt http://imada.sdu.dk/~duck/bishape.c}
$$
This algorithm provides us with a plethora of statistics for shapes of RNA complexes
of fixed topological genus. To illustrate local and global uniformity, we display in 
Fig.~\ref{F:uniformshape} the multiplicities of shapes of genus $1$. Here  by local 
uniformity we mean that we can uniformly sample shape of RNA complexes with a 
fixed number of arcs.

%%%
%%%%%%%%%%%%%%%%%%%%%%%%%%%%%%%%%%%%%%%
%%%%%%%%%%%%%%%%%%%%%%%%%%%%%%%%%%%%%%%%%%%%%%%%%%%%%%%%%
%%%
\begin{figure}[H]
\begin{center}
\includegraphics[width=0.45\textwidth,height=0.35\textwidth]{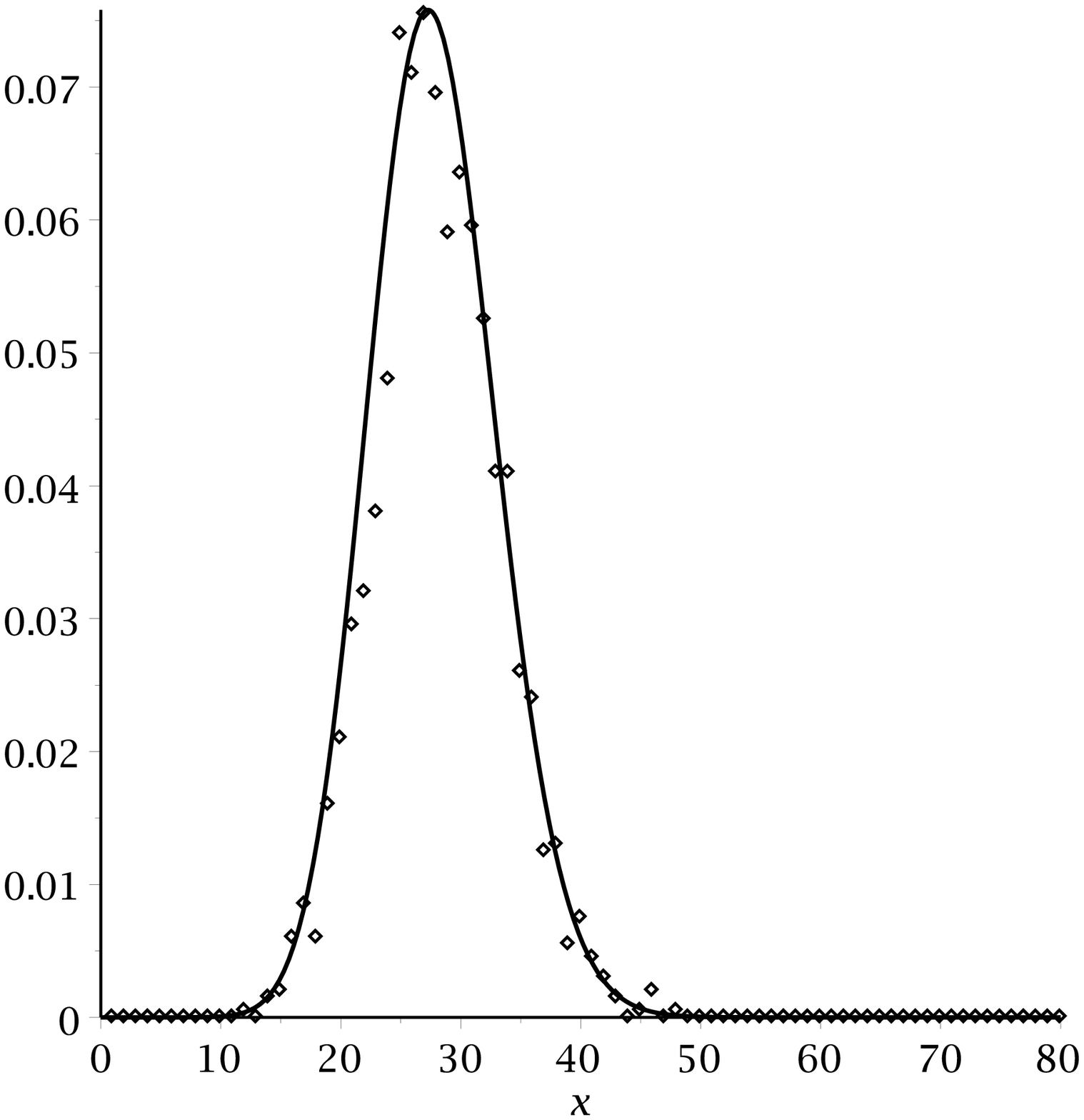}
\includegraphics[width=0.45\textwidth,height=0.35\textwidth]{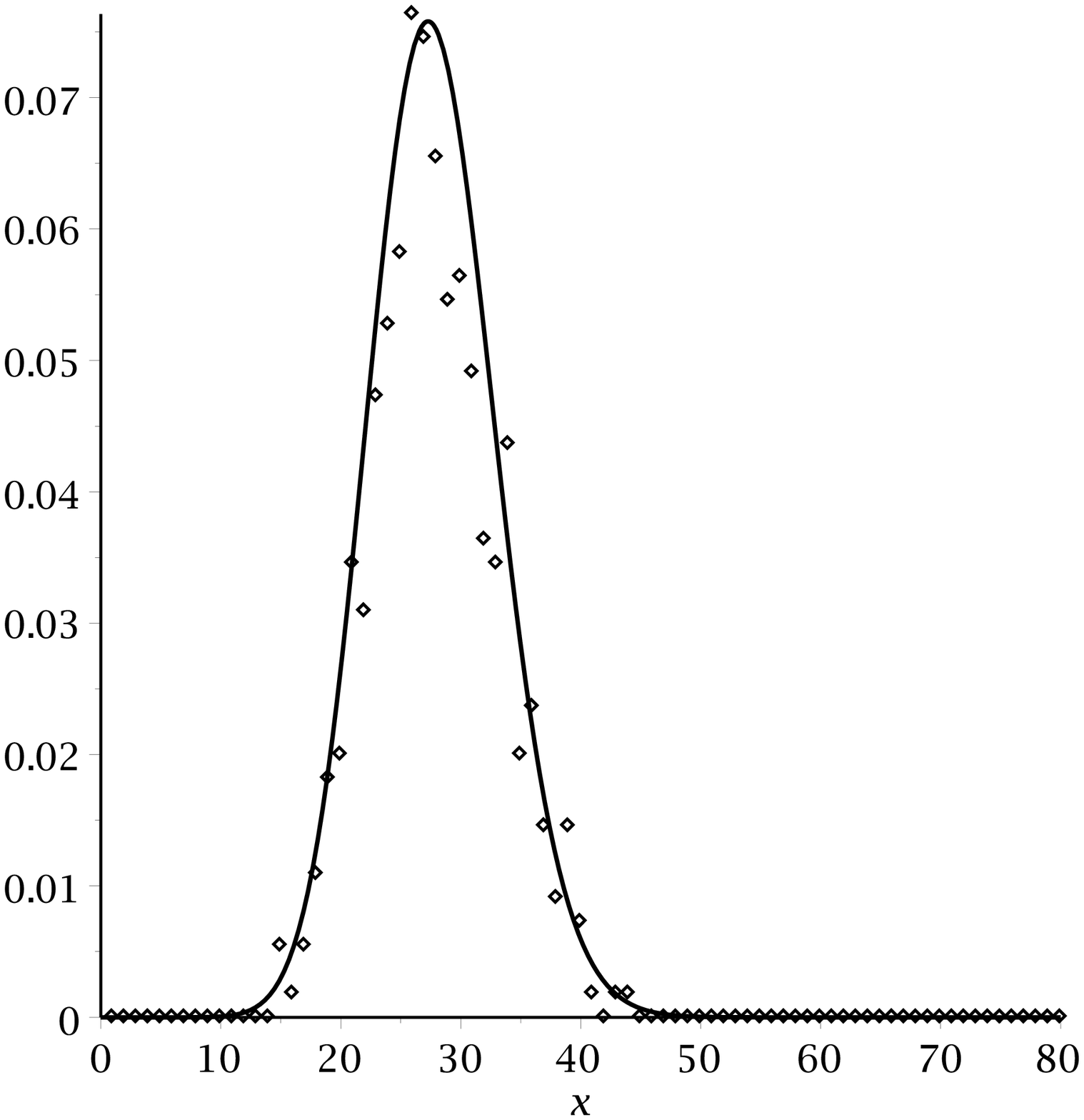}
\end{center}
\caption{ \small Global and local sampling of shapes of RNA complexes of fixed 
           topological genus:
$N=5\times 10^5$ shapes of genus $1$ were generated and we
display their multiplicities (dots) together with the binomial coefficients that are observed
from uniform sampling (LHS).
Local sampling: we generate $N=5\times 10^5$ shapes of genus $1$ 
with $7$ arcs(RHS).
}
\label{F:uniformshape}
\end{figure}
%%%
%%%%%%%%%%%%%%%%%%%%%%%%%%%%%%%%%%%%%%%

Lemma~\ref{L:finite} shows that there are only finitely many shapes of RNA complexes.
Hence the shape polynomial determines their numbers filtered by the number of arcs. 
This means that we can extract a finite observable from interaction structures that 
captures their topological core. 

Let us calibrate this information by inspecting what happens when we sample uniformly 
RNA complexes of fixed topological genus \citet{benjamin}.
We uniformly sample RNA complexes having genus $1$ and record the frequencies of their 
associated shapes. We observe that the distribution of shapes 
of different lengths equals the distribution obtained by normalizing the coefficients of 
the shape polynomial, see Fig.~\ref{F:coe}.

%%%
%%%%%%%%%%%%%%%%%%%%%%%%%%%%%%%%%%%%%%%
%%%%%%%%%%%%%%%%%%%%%%%%%%%%%%%%%%%%%%%%%%%%%%%%%%%%%%%%%
%%%
\begin{figure}[H]
\begin{center}
\includegraphics[width=0.5\textwidth,height=0.45\textwidth]{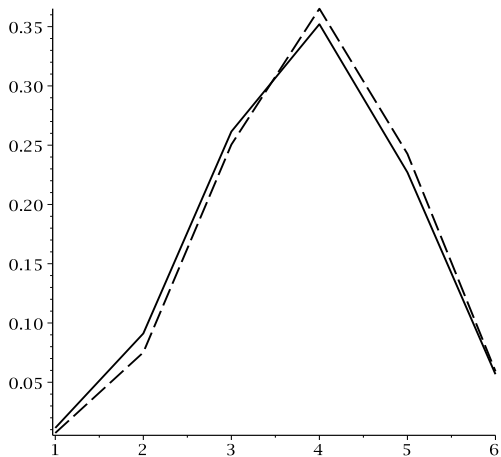}
\end{center}
\caption{ \small Uniform sampling of RNA complexes of genus $1$ with 
length $40,80,100,150,200$, ($5\times 10^5$). 
The solid curve displays the distribution induced by the coefficients of the shape 
polynomial, while the dashed curve displays distribution obtained from the sampling.
Displayed is the average of the coefficients obtained from sampling the above different
lengths.
}
\label{F:coe}
\end{figure}
%%%
%%%%%%%%%%%%%%%%%%%%%%%%%%%%%%%%%%%%%%%%%%%%%%%%%%%%%%%%%%
%%%

Accordingly, the shape polynomial represents precisely the uniform case. As a result
we can now compute the shapes of databases of RNA complexes and derive empirical
coefficients (distributions) and hence extract finite information from
databases reflecting the topological properties of the biological complexes.

Along these lines we study the shapes of biological RNA complexes 
obtained from \citet{Richter}. Due to the fact that the data set contained only exterior
arcs we derived only one shape of genus zero, see Fig.~\ref{F:datashape}.
  %%%%%%%%%%%%%%%%%%%%%%%%%%%%%%%%%%%%%%%
   %%%%%%%%%%%%%%%%%%%%%%%%%%%%%%%%%%%%%%%%%%%%%%%%%%%%%%%%%
   %%%
   \begin{figure}[H]
   \begin{center}
   \includegraphics[width=0.25\textwidth,height=0.15\textwidth]{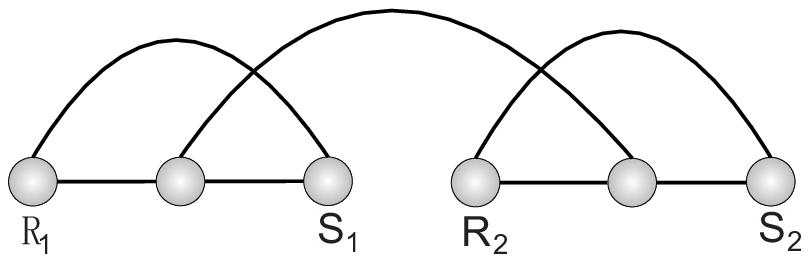}
   \end{center}
   \caption{ \small The shape extracted from the 
   biological RNA complexes \citet{Richter}.
   }
   \label{F:datashape}
   \end{figure}
   %%%%
   %%%%%%%%%%%%%%%%%%%%%%%%%%%%%%%%%%%%%%%%%%%%%%%
   %%%%
We accordingly compare the distribution of the exterior stack lengths of biological
with that of uniformly sampled RNA complexes, see Fig.~\ref{F:stacklength2}. 
%%
 %%%%%%%%%%%%%%%%%%%%%%%%%%%%%%%%%%%%%%%
 %%%%%%%%%%%%%%%%%%%%%%%%%%%%%%%%%%%%%%%%%%%%%%%%%%%%%%%%%
 %%%
 \begin{figure}[H]
 \begin{center}
 \includegraphics[width=0.45\textwidth,height=0.45\textwidth]{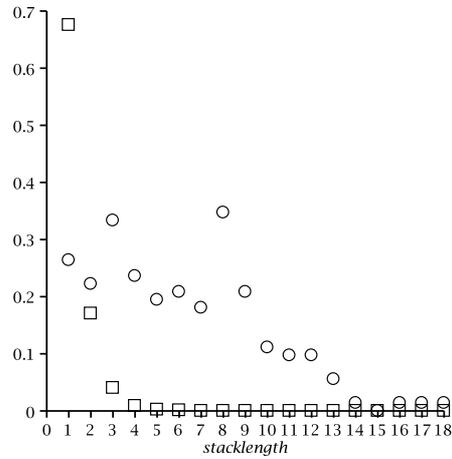}

 \end{center}
 \caption{ \small the distribution of the lengths of exterior stacks in uniformly
 sampled structures having the shape in Fig.~\ref{F:datashape} (box); 
 the distribution of the length of exterior stacks in the biological RNA complexes 
obtained from \citet{Richter} (circle).
 }
 \label{F:stacklength2}
 \end{figure}
 %%%%
 %%%%%%%%%%%%%%%%%%%%%%%%%%%%%%%%%%%%%%%%%%%%%%%
 %%%%

%%%
%%%%%%%%%%%%%%%%%%%%%%%%%%%%%%%%%%%%%%%%%%%%%%%%%%%%%%%%%%%%%%%%%%%%%%%%%%%%%%
%%%

We finally study loops in shapes of RNA complexes.
By construction such loops are multiloops, except of the two rainbow loops.
We uniformly generate $5\times 10^5$ shapes of RNA complexes from
genus $0$ to $5$ and display the average number of loops, see Fig.~\ref{F:shapeloop}. 
The data suggest a central limit theorem for the average number of loops since their
mean scales linearly with topological genus.
%%%
 %%%%%%%%%%%%%%%%%%%%%%%%%%%%%%%%%%%%%%%
 %%%%%%%%%%%%%%%%%%%%%%%%%%%%%%%%%%%%%%%%%%%%%%%%%%%%%%%%%
 %%%
 \begin{figure}[H]
 \begin{center}
 \includegraphics[width=0.45\textwidth,height=0.45\textwidth]{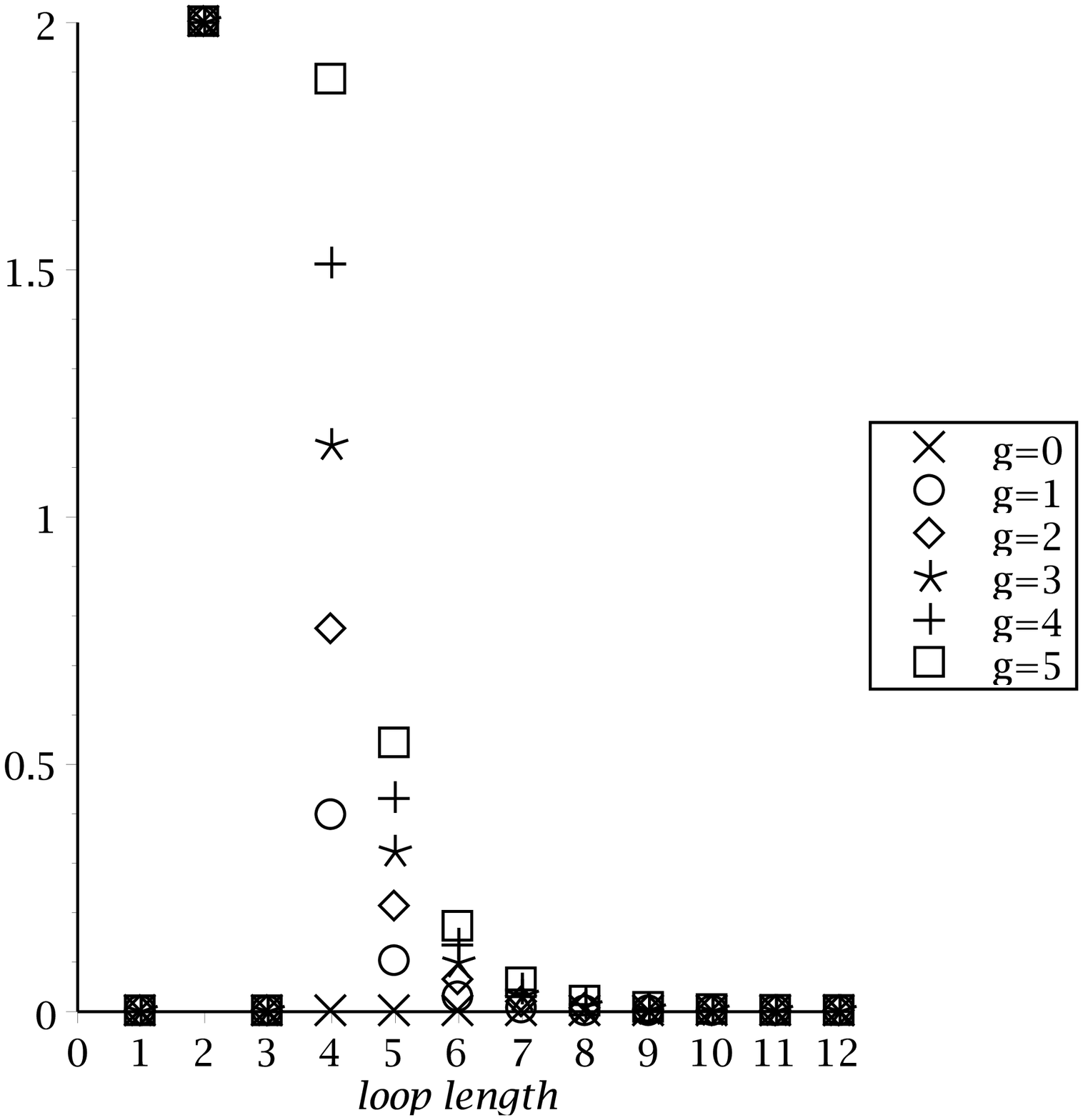}
  \includegraphics[width=0.45\textwidth,height=0.45\textwidth]{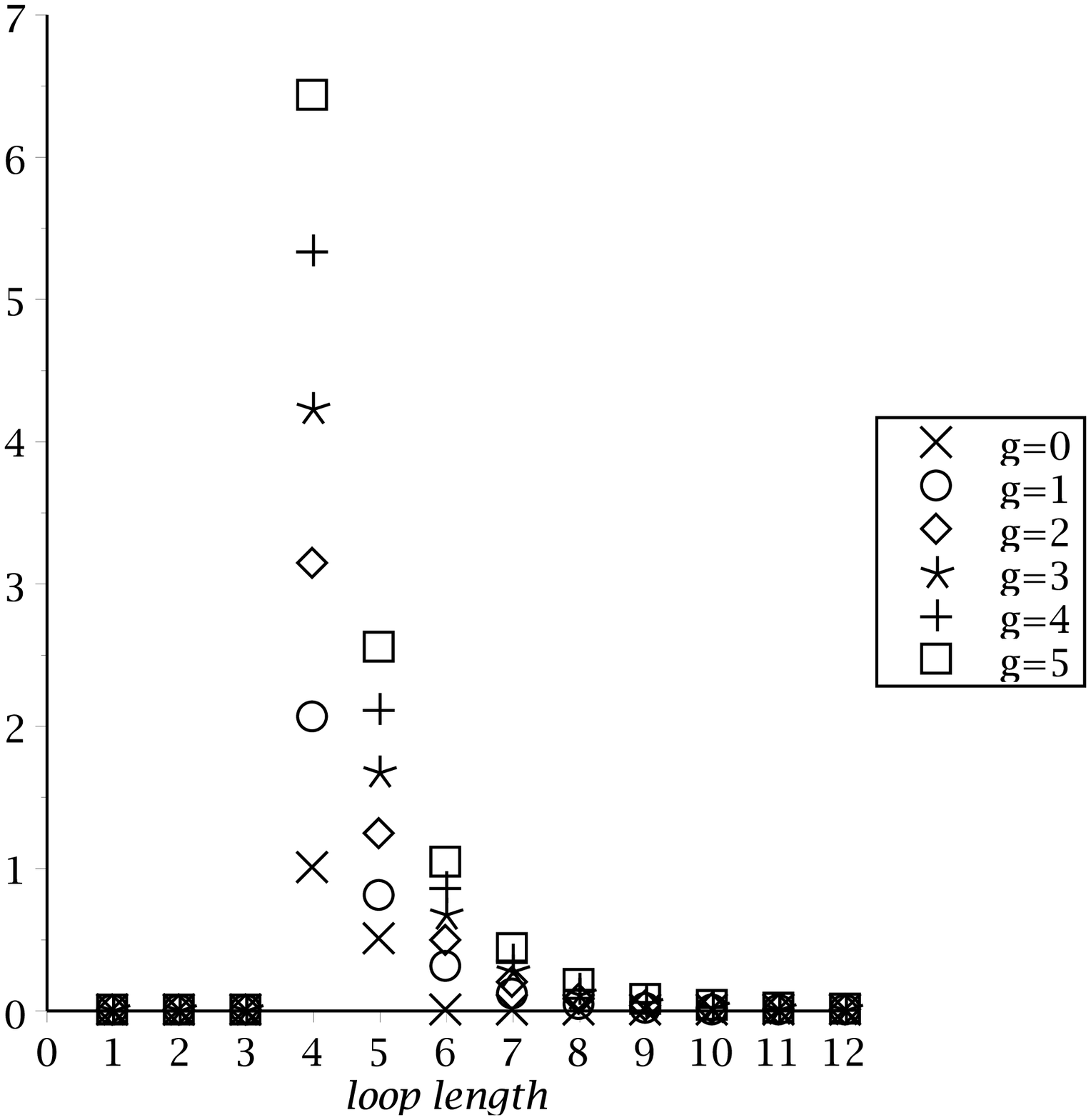}
 \end{center}
 \caption{ \small  The distribution of the average number of loops 
 in the shapes of different genus: the distribution of the $\alpha$-loops 
(loops contained in one backbone) (LHS).
The distribution of the $\beta$-loops (loops over two backbones) (RHS).
 }
 \label{F:shapeloop}
 \end{figure}
 %%%
 %%%%%%%%%%%%%%%%%%%%%%%%%%%%%%%%%%%%%%%
\section*{Acknowledgements}\label{sec:ack}
We wish to thank Fenix W.D. Huang and Thomas J.X. Li for discussions.
This work is funded by the Future and Emerging Technologies (FET) programme 
of the European Commission within the Seventh Framework Programme (FP7), 
under the FET-Proactive grant agreement TOPDRIM, FP7-ICT-318121.

\section*{Author Disclosure Statement}
No competing financial interests exist.

 %%%
 %%%%%%%%%%%%%%%%%%%%%%%%%%%%%%%%%%%%%%%
\bibliographystyle{plainnat}
\bibliography{reference}

\begin{thebibliography}{32}
\providecommand{\natexlab}[1]{#1}
\providecommand{\url}[1]{\texttt{#1}}
\expandafter\ifx\csname urlstyle\endcsname\relax
  \providecommand{\doi}[1]{doi: #1}\else
  \providecommand{\doi}{doi: \begingroup \urlstyle{rm}\Url}\fi

\bibitem[Andersen et~al.(2012)Andersen, Huang, Penner, and Reidys]{Fenix:t2}
J{\o}rgen~E. Andersen, Fenix~W.D. Huang, Robert~C. Penner, and Christian~M
  Reidys.
\newblock Topology of {RNA}-{RNA} interaction structures.
\newblock \emph{Journal of Computational Biology}, 19\penalty0 (7):\penalty0
  928--943, 2012.

\bibitem[Andersen et~al.(2011)Andersen, Penner, Reidys, and
  Waterman]{andersen2011}
J{\o}rgen~Ellegaard Andersen, Robert~C Penner, Christian~M Reidys, and
  Michael~S Waterman.
\newblock Topological classification and enumeration of {RNA} structures by
  genus.
\newblock \emph{Journal of mathematical biology}, pages 1--18, 2011.

\bibitem[Bachellerie et~al.(2002)Bachellerie, Cavaill{\'e}, and
  H{\"u}ttenhofer]{Bachellerie}
Jean-Pierre Bachellerie, J{\'e}r{\^o}me Cavaill{\'e}, and Alexander
  H{\"u}ttenhofer.
\newblock The expanding sno{RNA} world.
\newblock \emph{Biochimie}, 84\penalty0 (8):\penalty0 775--790, 2002.

\bibitem[Banerjee and Slack(2002)]{Banerjee}
Diya Banerjee and Frank Slack.
\newblock Control of developmental timing by small temporal {RNA}s: a paradigm
  for {RNA}--mediated regulation of gene expression.
\newblock \emph{Bioessays}, 24\penalty0 (2):\penalty0 119--129, 2002.

\bibitem[Benne(1989)]{Benne}
Rob Benne.
\newblock {RNA}--editing in trypanosome mitochondria.
\newblock \emph{Biochimica et Biophysica Acta (BBA)-Gene Structure and
  Expression}, 1007\penalty0 (2):\penalty0 131--139, 1989.

\bibitem[Bon et~al.(2008{\natexlab{a}})Bon, Vernizzi, Orland, and Zee]{Bon:08}
Michael Bon, Graziano Vernizzi, Henri Orland, and A.~Zee.
\newblock Topological classification of {RNA} structures.
\newblock \emph{J.\ Mol.\ Biol.}, 379:\penalty0 900--911, 2008{\natexlab{a}}.

\bibitem[Bon et~al.(2008{\natexlab{b}})Bon, Vernizzi, Orland, and Zee]{bon}
Michael Bon, Graziano Vernizzi, Henri Orland, and A~Zee.
\newblock Topological classification of {RNA} structures.
\newblock \emph{Journal of molecular biology}, 379\penalty0 (4):\penalty0
  900--911, 2008{\natexlab{b}}.

\bibitem[Chapuy(2011)]{Chapuy:11}
Guillaume Chapuy.
\newblock A new combinatorial identity for unicellular maps, via a direct
  bijective approach.
\newblock \emph{Adv. Appl. Math.}, 47(4):\penalty0 874--893, 2011.

\bibitem[Chekhov(1997)]{chekhov}
L~Chekhov.
\newblock Matrix model tools and geometry of moduli spaces.
\newblock \emph{Acta Applicandae Mathematica}, 48\penalty0 (1):\penalty0
  33--90, 1997.

\bibitem[Flajolet and Sedgewick(2009)]{flajolet}
Philippe Flajolet and Robert Sedgewick.
\newblock \emph{Analytic combinatorics}.
\newblock cambridge University press, 2009.

\bibitem[Fu et~al.(2013)Fu, Han, and Reidys]{benjamin}
Benjamin~M.M. Fu, Hillary~S.W. Han, and Christian~M Reidys.
\newblock On the {RNA}-{RNA} interaction structures of fixed topological genus.
\newblock \emph{arXiv:1311.0684v2}, 2013.

\bibitem[Harer and Zagier(1986)]{Harer}
J.~Harer and D.~Zagier.
\newblock The euler characteristic of the moduli space of curves.
\newblock \emph{Invent.Math.}, 85:\penalty0 457--486, 1986.

\bibitem[Huang and Reidys(2014)]{fenix-shape}
Fenix~W.D. Huang and Christian~M Reidys.
\newblock Shapes of topological {RNA} structures.
\newblock \emph{arXiv:1403.2908}, 2014.

\bibitem[Kleitman(1970)]{Kleitman:70}
D.~Kleitman.
\newblock Proportions of {I}rreducible {D}iagrams.
\newblock \emph{Studies in Appl. Math.}, 49:\penalty0 297--299, 1970.

\bibitem[Kugel and Goodrich(2007)]{Kugel}
Jennifer~F. Kugel and James~A. Goodrich.
\newblock An {RNA} transcriptional regulator templates its own regulatory
  {RNA}.
\newblock \emph{Nature chemical biology}, 3\penalty0 (2):\penalty0 89--90,
  2007.

\bibitem[Li and Reidys(2014)]{thomas}
Thomas Li and Christian Reidys.
\newblock personal communication, 2014.

\bibitem[Massey(1967)]{Massey:69}
William~S. Massey.
\newblock \emph{Algebraic Topology: An Introduction}.
\newblock Springer-Veriag, New York., 1967.

\bibitem[McManus and Sharp(2002)]{McManus}
Michael~T. McManus and Phillip~A. Sharp.
\newblock Gene silencing in mammals by small interfering {RNA}s.
\newblock \emph{Nature reviews genetics}, 3\penalty0 (10):\penalty0 737--747,
  2002.

\bibitem[Narberhaus and Vogel(2007)]{Vogel:07}
Franz Narberhaus and J{\"o}rg Vogel.
\newblock Sensory and regulatory {RNA}s in prokaryotes: {A} new german research
  focus.
\newblock \emph{RNA biology}, 4\penalty0 (3):\penalty0 160--164, 2007.

\bibitem[Nussinov et~al.(1978)Nussinov, Pieczenik, Griggs, and
  Kleitman]{Nussinov:1978}
Ruth Nussinov, George Pieczenik, Jerrold~R Griggs, and Daniel~J Kleitman.
\newblock Algorithms for loop matchings.
\newblock \emph{SIAM Journal on Applied mathematics}, 35\penalty0 (1):\penalty0
  68--82, 1978.

\bibitem[Orland and Zee(2002)]{Orland:02}
Henri Orland and A.~Zee.
\newblock {RNA} folding and large {N} matrix theory.
\newblock \emph{Nuclear Physics B}, 620\penalty0 (3):\penalty0 456--476, 2002.

\bibitem[Penner(2004)]{Penner:03}
R.~C. Penner.
\newblock Cell decomposition and compactification of {R}iemann's moduli space
  in decorated {T}eichm{\"u}ller theory.
\newblock In Nils Tongring and R.~C. Penner, editors, \emph{Woods Hole
  Mathematics-perspectives in math and physics}, pages 263--301. World
  Scientific, Singapore, 2004.
\newblock arXiv: math.GT/0306190.

\bibitem[Penner et~al.(2010)Penner, Knudsen, Wiuf, and Andersen]{protein}
R.~C. Penner, Michael Knudsen, Carsten Wiuf, and J{\o}rgen~Ellegaard Andersen.
\newblock Fatgraph models of proteins.
\newblock \emph{Comm.\ Pure Appl.\ Math.}, 63:\penalty0 1249--1297, 2010.

\bibitem[Penner and Waterman(1993)]{Waterman:93}
R.C. Penner and Michael~S. Waterman.
\newblock Spaces of {RNA} secondary structures.
\newblock \emph{Advances in Mathematics}, 101\penalty0 (1):\penalty0 31--49,
  1993.

\bibitem[Reidys et~al.(2011)Reidys, Huang, Andersen, Penner, Stadler, and
  Nebel]{Reidys:11a}
C.~M. Reidys, F.~Huang, J.~E. Andersen, R.~C. Penner, P.~F. Stadler, and M.~E.
  Nebel.
\newblock Topology and prediction of {RNA} pseudoknots.
\newblock \emph{Bioinformatics}, 27:\penalty0 1076--1085, 2011.

\bibitem[Reidys et~al.(2010)Reidys, Wang, and Zhao]{Reidys:10w}
Christian~M Reidys, Rita~R Wang, and Albus~YY Zhao.
\newblock Modular, k-noncrossing diagrams.
\newblock \emph{the electronic journal of combinatorics}, 17\penalty0
  (R76):\penalty0 1, 2010.

\bibitem[Richter and Backofen(2012)]{Richter}
Andreas~S. Richter and Rolf Backofen.
\newblock Accessibility and conservation: General features of bacterial small
  {RNA}--m{RNA} interactions?
\newblock \emph{RNA Biology}, 9:\penalty0 954--965, 2012.

\bibitem[Rivas and Eddy(1999)]{Rivas:99}
Elena Rivas and Sean~R. Eddy.
\newblock A dynamic programming algorithm for {RNA} structure prediction
  including pseudoknots.
\newblock \emph{J.\ Mol.\ Biol.}, 285:\penalty0 2053--2068, 1999.

\bibitem[Vernizzi and Orland(2005)]{vernizzi}
Graziano Vernizzi and Henri Orland.
\newblock Large--{N} random matrices for {RNA} folding.
\newblock \emph{Acta PhysicA PolonicA Series B}, 36\penalty0 (9):\penalty0
  2821, 2005.

\bibitem[Vernizzi et~al.(2005)Vernizzi, Orland, and Zee]{vernizzib}
Graziano Vernizzi, Henri Orland, and A~Zee.
\newblock Enumeration of {RNA} structures by matrix models.
\newblock \emph{Physical review letters}, 94\penalty0 (16):\penalty0 168103,
  2005.

\bibitem[Waterman(1978{\natexlab{a}})]{Waterman:79a}
Michael~S. Waterman.
\newblock Combinatorics of {RNA} hairpins and cloverleaves.
\newblock \emph{Studies Appl. Math}, 60:\penalty0 91--96, 1978{\natexlab{a}}.

\bibitem[Waterman(1978{\natexlab{b}})]{waterman1978secondary}
Michael~S. Waterman.
\newblock Secondary structure of single--stranded nucleic acids.
\newblock \emph{Adv. math. suppl. studies}, 1:\penalty0 167--212,
  1978{\natexlab{b}}.

\end{thebibliography}

%%%
%%%%%%%%%%%%%%%%%%%%%%%%%%%%%%%%%%%%%%%%%%%%%%%%%%%%%%%%%%%%%%%%%%%%%%%%%
%%%
\newpage
\begin{table}
\begin{center}
\begin{tabular}{c|ccccc}
 & g=1 & 2 & 3 & 4 & 5 \\
\hline
  t=0 & 1 & 21 & 1485 & 225225 & 59520825 \\
  1 &  & 105 & 18018 & 4660227 & 1804142340 \\
  2 &  &  & 50050 & 29099070 & 18472089636 \\
  3 &  &  &  & 56581525 & 78082504500 \\
  4 &  &  &  &  &  117123756750
\end{tabular}
\end{center}
\caption{\small The coefficients $\kappa_t^{(g)}$.
}\label{T:kappa}
\end{table}
%%%
%%%
%%%
%%%%%%%%%%%%%%%%%%%%%%%%%%%%%%%%%%%%%%%%%%%%%%%%%%%%%%%%%%%%%%%%%%%%%%%%%%%%%%
%%%

\end{document}